\documentclass[authoryear,preprint,3p,review,12pt]{elsarticle}

\usepackage{amsmath}
\usepackage{amssymb}
\usepackage{booktabs}
\usepackage{amsfonts}
\usepackage{multirow}
\usepackage{amsthm}
\usepackage{bbm}
\usepackage{xcolor,float}
\usepackage{caption}
\usepackage{subcaption}
\usepackage[hidelinks]{hyperref}

\usepackage{graphicx}
\usepackage{slashbox}
\usepackage{natbib}

\usepackage{todonotes}
\usepackage{comment}
\newtheorem{theorem}{Theorem}
\newtheorem{lemma}{Lemma}
\newtheorem{proposition}{Proposition}
\newtheorem{corollary}{Corollary}

\theoremstyle{definition}

\newtheorem{definition}{Definition}

\newcommand{\R}{\mathbb{R}}

\newcommand{\E}{\mathbb{E}}
\newcommand{\F}{\mathbb{F}}

\newcommand{\G}{\mathbb{G}}

\begin{document}
	\makeatletter
	\def\ps@pprintTitle{%
		\let\@oddhead\@empty
		\let\@evenhead\@empty
		\let\@oddfoot\@empty
		\let\@evenfoot\@oddfoot
	}
	\makeatother

	\begin{frontmatter}
		
		

		\author[label1]{Tommaso Lando\corref{c1}} 
	\fntext[label1]{Department of Economics, University of Bergamo, via dei Caniana 2, 24127, Bergamo, Italy. orcid:0000-0003-4288-0264}
	\ead{tommaso.lando@unibg.it}
	
	\author[label1]{Mohammed Es-Salih Benjrada} 


\cortext[c1]{Corresponding author}

\title{A new class of tests for convex-ordered families based on expected order statistics}




\begin{abstract}Consider a pair of cumulative distribution functions $F$ and $G$, where $F$ is unknown and $G$ is a known reference distribution. Given a sample from $F$, we propose tests to detect the convexity or the concavity of $G^{-1}\circ F$ versus equality in distribution (up to location and scale transformations). This framework encompasses well-known cases, including increasing hazard rate distributions, as well as some other relevant families that have garnered attention more recently, for which no tests are currently available. We introduce test statistics based on the estimated probability that the random variable of interest does not exceed a given expected order statistic, which, in turn, is estimated via L-estimation. The tests are unbiased, consistent, and exhibit monotone power with respect to the convex transform order. To ensure consistency, we show that our L-estimators satisfy a strong law of large numbers, even when the mean is not finite, thereby making the tests suitable for heavy-tailed distributions. Unlike other approaches, these tests are broadly applicable, regardless of the choice of $G$ and without support restrictions. The performance of the method under various conditions is demonstrated via simulations, and its applicability is illustrated through a concrete example.
\end{abstract}



\end{frontmatter}
\section{Introduction}\label{intro}
An interesting problem in nonparametric statistics is testing whether the cumulative distribution function (CDF) of interest $F$ belongs to some given class. Given a reference absolutely continuous CDF $G$, which we assume to be known, many important families can be defined as $\mathcal{F}^{cx}_{G}=\{F:G^{-1}\circ F\text{ is convex}\}$ or $\mathcal{F}^{cv}_{G}=\{F:G^{-1}\circ F\text{ is concave}\}$. We may refer to $\mathcal{F}_G^{cx}$ and $\mathcal{F}_G^{cv}$ as \textit{convex-ordered} families, since $F\in\mathcal{F}^{cx}_G$ means that $F$ is less than $G$ in the \textit{convex transform order} \citep{barlowzwet,shaked,vanzwet1964}, while the order is reversed if $F\in\mathcal{F}^{cv}_G$. Relevant examples are: the families of convex and concave CDFs \citep{grenander}, if $G$ is the uniform; the increasing and decreasing hazard rate (IHR, DHR) families \citep{LifeDist,shaked}, obtained when $G$ is exponential; the decreasing reversed hazard rate (DRHR) class \citep{barlow1963,block,LifeDist}, obtained when $1-G(-\cdot)$ is exponential; the increasing and decreasing odds rate family (IOR, DOR) \citep{lando2024,odds}, obtained when $G$ is a log-logistic with shape parameter equal to 1 (DOR models are denoted as ``super-Pareto" in \cite{chen2024}); the increasing log-odds rate family (ILOR) \citep{zimmer}, obtained when $G$ is a logistic distribution; the ``super-Fréchet" and the ``super-Cauchy" classes \citep{chen2024stochastic,muller2024}, obtained when $G$ is a Fréchet or a Cauchy distribution, respectively. These classes have different mathematical properties that can be used, for instance, in decision theory, reliability, and survival analysis. Moreover, from a statistical perspective, if we know that $F$ belongs to some convex-ordered family, we can use this information to improve the nonparametric estimate of $F$. This approach corresponds to the scope of shape-constrained inference; see, for instance, the books by \cite{robertson1988} and \cite{groeneboom2014}. 

{For these reasons, nonparametric tests for convex-ordered families are particularly interesting and have been studied extensively; see for instance \cite{anis,beare2021,bickel,carolan,gibels,groeneboom2012,hall2005,lando2023,lando2024,landotransform,mitra2008,proschan1967}. However, despite the wide availability of tests for the more well-known cases, particularly the IHR family, several interesting classes that have been studied recently, such as some of those mentioned earlier, currently lack any existing} tests. Moreover, the available tests typically rely on a specific choice of $G$ and they function properly under assumptions on the supports of $F$ and $G$, sometimes quite restrictive. For example, the established test of \cite{proschan1967}, as well as those of \citep{bickel,gibels}, are based on a stochastic ordering property of sample spacings which holds only in the IHR and DHR cases; that is, they are valid specifically when $G$ is exponential. Similar limitations apply to other tests. We also emphasize that in some of the cited tests, the shape property is treated as the null hypothesis — an approach typically referred to as goodness-of-fit — while in others, it is treated as the alternative hypothesis. These approaches are somewhat complementary; we adopt the latter in this paper. Recalling that the convex order is location and scale-invariant, denote with $\mathcal{G}$ the location-scale family obtained from $G$. We propose a general method to test the null hypothesis $$\mathcal{H}_0^{G}:F\in \mathcal{G},$$ versus the alternatives $$\mathcal{H}_{1+}^{G}:F\in\mathcal{F}^{cx}_G- \mathcal{G}\quad\text{ and }\quad\mathcal{H}_{1-}^G:F\in\mathcal{F}^{cv}_G- \mathcal{G}.$$ This approach can be used in combination with goodness-of-fit tests. If, for a given dataset, $\mathcal{H}_0^{G}$ is rejected in favour of $\mathcal{H}_{1+}^{G}$ ($\mathcal{H}_{1-}^{G}$), whereas a goodness-of-fit test for $F\in\mathcal{F}^{cx}_G$ ($F\in\mathcal{F}^{cv}_G$) fails to reject the null, this suggests that $\mathcal{H}_{1+}^{G}$  can indeed hold.
Differently from other approaches, the proposed method works for any choice of $G$ and does not require a finite mean. Moreover, we do not impose any limitations on the supports of $F$ and $G$.

The article is organised as follows. In Section \ref{section2}, we present some preliminary notions and results. Our approach is mainly based on the following property of expected order statistics, which can be easily derived using Jensen's inequality. If $F$ belongs to a convex-ordered family, the probability of not exceeding a given expected order statistic from $F$ is bounded, according to $G$. To verify whether such bounds are empirically satisfied, we need estimators of the expected order statistics. Hence, in Section \ref{section3}, we introduce L-estimators of the expected order statistics and study their properties. In particular, we prove that, under some conditions related to the tail behaviour, these satisfy a strong law of large numbers, even when the distribution does not have a finite mean. These estimators can be used to estimate the probability of exceeding the bounds. Consequently, in Section \ref{section4}, we propose a class of test statistics based on the distance between the aforementioned bounds and the sample counterpart of the non-exceedance probabilities. We also establish the theoretical properties of our family of tests. In particular, we show that the tests are unbiased and have monotone power for every fixed sample size. Moreover, we establish the consistency of our tests in the case of finite and infinite, or undefined, mean. This property ensures that our tests are well-suited for handling distributions with heavy tails. The behaviour of the tests from a practical point of view is established via simulations, in Section \ref{section5}. In the IHR and DHR cases, we compare our tests with the well-known test of \cite{proschan1967}, which has been shown to satisfy the same theoretical properties of our class of tests \citep{bickel}. The result shows that our tests are generally outperformed in terms of power when the distribution has a monotone hazard rate. However, when the hazard rate is non-monotone, our tests (applied to both the IHR and DHR properties) are capable of detecting this, whereas the test of \cite{proschan1967} may lead to the wrong decision. We also study the IHR test of \cite{mitra2008} (see also \cite{anis2014}), which exhibits large power but suffers from two main limitations: it requires non-negative support and lacks location-invariance.
Moreover, we apply our tests to other important families of distributions for which the approach of \cite{proschan1967} is unsuitable and no tests are currently available, such as the IOR, DOR, and DRHR classes.  In all the scenarios considered, the simulations confirm the theoretical properties established. Finally, in Section \ref{section6}, we provide an example where our tests are applied to river flow data, to examine the shape properties of the underlying distribution, with particular emphasis on its tail behaviour. The proofs of our results are reported in Appendix \ref{app_proofs}.
\section{Preliminaries}\label{section2}
\subsection{Notations}
In this paper, increasing and decreasing mean non-decreasing and non-increasing, respectively. The generalised inverse of an increasing function $u$ is $u^{-1}(y)=\inf\{x:u(x)\geq y\}$. The positive and negative parts of a function $v$ are defined as $v_+=\max(v,0)$ and $v_-=-\min(v,0)$, respectively. The $L^p$ norm of an $m$-dimensional vector $\mathbf{z}=(z_1,\ldots,z_m)$ is defined as $||\mathbf{z}||_p=(|z_1|^p+\ldots+|z_m|^p)^{1/p}$, for $p\in[1,\infty),$ while $||\mathbf{z}||_\infty=\max\{|z_1|,\ldots,|z_m|\}$.
Let $X$ be a random variable with cumulative distribution function (CDF) $F$. The mean of the distribution $F$, whenever it is defined, will be denoted with $\E X=\mu_F=\mu(F),$ as more convenient. Let us denote with $X_{j:m}$ the $j$-th order statistic corresponding to a random sample $X_1,\ldots,X_m$ of size $m$ from $X$. It is well-kown that the CDF of $X_{j:m}$ is $F_{B_{j:m}}\circ F,$ where $F_{B_{j:m}}$ represents the CDF of a beta random variable $B_{j:m},$ with parameters $j$ and $m-j+1.$ The density of $F_{B_{j:m}}$ is denoted with $f_{B_{j:m}}$.

We shall use the following stochastic orders \citep{shaked}. 
\begin{definition}\label{def1}Given a pair of CDFs $F$ and $G$, we say that $F$ is less than $G$
	\begin{enumerate}
		
		\item in the usual stochastic order, denoted as $F\leq_{st}G$, if $F(x)\geq G(x),$ for every $x\in\R;$
		\item in the convex transform order, denoted as $F\leq_c G$, if $G^{-1}\circ F$ is convex.
	\end{enumerate}
\end{definition}
Since stochastic orders depend only on the distributions of the random variables, for an order $\succ$ and a pair of random variables $X$ and $Y$ with CDFs $F$ and $G$, respectively, the notations $X\succ Y$ and $F\succ G$ will be used interchangeably. Finally, the symbol $=_d$ denotes equality in distribution.

\subsection{Some properties of convex-ordered families}
Since any increasing convex function has an increasing concave inverse, and vice-versa, it is clear from Definition \ref{def1} that $\mathcal{F}_G^{cx}=\{F:F\leq_c G\}$ and $\mathcal{F}_G^{cv}=\{F:G\leq_c F\}$. If $F$ and $G$ both belong to $\mathcal{G},$ then the composition $G^{-1}\circ F$ is affine, so that $F\leq_c G$ and $G\leq_c F,$ in other words, the distributions are equivalent with respect to the convex transform order. Hereafter, we will assume that $G$ is absolutely continuous, with density $g$. However, the convexity (or concavity) of the composition $G^{-1}\circ F$ implies that $F$ has a density $f$ almost everywhere, except, possibly, at the right or left endpoints of its support; see, for example, Proposition C.2 in \cite{LifeDist}. The families $\mathcal{F}_G^{cx}$ and $\mathcal{F}_G^{cv}$ are also related to the notion of \textit{generalised hazard rate} \citep{barlowzwet,robertson1988}. Indeed, one may define the generalised hazard rate as the derivative of $G^{-1}\circ F$, which is $$h^G_F(x)=\frac d{dx}G^{-1}\circ F(x)=\frac{f(x)}{g\circ G^{-1}\circ F(x)},$$
so that $F\in\mathcal{F}_G^{cx}$ or $F\in\mathcal{F}_G^{cv}$ if $h^G_F$ is increasing, or decreasing, respectively. By a change of variables, $F\in\mathcal{F}_G^{cx}$ or $F\in\mathcal{F}_G^{cv}$ if and only if the ratio $f\circ F^{-1}/g\circ G^{-1}$, namely, the ratio between \textit{quantile densities} \citep{staudte}, is increasing, or decreasing, respectively. This means that $F\in\mathcal{F}_G^{cx}$ or $F\in\mathcal{F}_G^{cv}$ if and only if $$\int_0^p\frac{f\circ F^{-1}(u)} {g\circ G^{-1}(u)}du=\int_0^{F^{-1}(p)}g\circ G^{-1}\circ F(t)dt,\quad p\in[0,1],$$ which is referred to as the \textit{generalised TTT transform} of $F$ with respect to $G$ \citep{robertson1988}, is concave, or convex, respectively. It is also interesting to note that the shape of the composition $G^{-1}\circ F$ is often related to the notions of skewness and kurtosis \citep{oja,staudte2017}.

The conditions $F\in\mathcal{F}_G^{cx}$ and $F\in\mathcal{F}_G^{cv}$ have an effect on the expectation of $X_{j:m}$, denoted with $\E X_{j:m}=\mu_{j:m}(F)=\mu_{j:m}$. Indeed, for $F\in\mathcal{F}_G^{cx}$, Jensen's inequality implies that
$$\E X_{j:m}\leq F^{-1}\circ G(\E(G^{-1}\circ F(X_{j:m})))=F^{-1}\circ G(\E(G^{-1}( B_{j:m}))),$$
therefore, $P(X\leq \mu_{j:m})\leq G(\E(G^{-1}(B_{j:m})).$
Similar results hold when $F\in\mathcal{F}_G^{cv}$. We summarise this property, presented in \cite{landoJAP}, as follows.
\begin{proposition}\label{prop}Let $\pi_{j:m}^G=G(\E(G^{-1}(B_{j:m})))$. Then, for every $j,m$ such that $\E(G^{-1}(B_{j:m}))$ is defined,
	\begin{enumerate}
		\item 	
		if $F\in\mathcal{F}_G^{cx}$, then $P(X\leq \mu_{j:m})\leq \pi_{j:m}^G$;
		\item 	
		if $F\in\mathcal{F}_G^{cv}$, then $P(X\leq \mu_{j:m})\geq \pi_{j:m}^G$.
	\end{enumerate}
\end{proposition}

The existence of a finite mean ensures that the expected order statistics $\mu_{j:m}$ also exist and are finite, for every $j$ and $m$. However, this is not necessary; we may have $\E |X|=\infty$ but $\mu_{j:m}<\infty$ for some $j$ and $m$, as discussed in the next section. In any case, Proposition~\ref{prop} always works, giving the trivial bounds 0 or 1 whenever the expected order statistics are $-\infty$ or $+\infty$, respectively. 

Basic choices of $G$ yield the following popular classes of distributions, with the corresponding explicit expressions of $\pi_{j:m}^G$.
\begin{enumerate}
	\item \textit{Uniform}. Let $G=U,$ where $U(x)=x$, for $x\in[0,1].$ Then $\mathcal{F}_U^{cx}$ and $\mathcal{F}_U^{cv}$ are the classes of convex and concave CDFs, respectively. Moreover, $\pi_{j:m}^U=\frac{j}{m
		+1}$.
	\item \textit{Exponential}. Let $G=\mathcal{E}$, where $\mathcal{E}(x)=1-e^{-x},$ for $x>0.$ Then, since $h_F^\mathcal{E}=f/(1-F)$ is the classic hazard rate (HR) of $F$ \citep{LifeDist}, then $\mathcal{F}_\mathcal{E}^{cx}$ and $\mathcal{F}_\mathcal{E}^{cv}$ are the classes of IHR and DHR distributions, respectively. In this case, $\pi_{j:m}^\mathcal{E}=1-\exp\left(-\sum_{k=m-j+1}^m \frac1k\right)$.
	\item \textit{Negative exponential}. Let $G=\mathcal{E_-}$, where $\mathcal{E}_-(x)=e^{x}$, for $x<0.$ Then, since $h_F^{\mathcal{E}_-}=f/F$ is the reversed hazard rate of $F$, $\mathcal{F}_\mathcal{E_-}^{cv}$ and $\mathcal{F}_\mathcal{E_-}^{cx}$ are the classes of IRHR and DRHR distributions \citep{block,LifeDist}. In this case, $\pi_{j:m}^\mathcal{E_-}=\exp\left(-\sum_{k=j}^m \frac1k\right)$.
	\item \textit{Log-logistic}. Let $G=\mathcal{L}$, where $\mathcal{L}(x)=x/(1+x),$ for $x>0.$ Then, since $h_F^{\mathcal{E}_-}=f/(1-F)^2$ is the derivative of the odds function $F/(1-F)$ \citep{odds}, $\mathcal{F}_\mathcal{L}^{cx}$ and $\mathcal{F}_\mathcal{L}^{cv}$ are the classes of IOR and DOR distributions, respectively. In this case, $\pi_{j:m}^{\mathcal{L}}=\frac{j}m$.
\end{enumerate}

Explicit formulas are not always available, however, $\pi_{j:m}^G$ can always be computed, regardless of $G$, since we assume $G$ to be known. Some other interesting examples may be obtained by letting $G$ be the CDF of the Fréchet or the Cauchy distributions, giving rise to families of heavy-tailed distributions that have been recently studied \citep{muller2024}.

A simple special case of Proposition~\ref{prop} is obtained for $j=m=1.$ 
\begin{corollary}For every $F$ and $G$ with defined mean,
	\begin{enumerate}
		\item 	
		if $F\in\mathcal{F}_G^{cx}$, then $F(\mu_F)\leq G(\mu_G)$;
		\item 	
		if $F\in\mathcal{F}_G^{cv}$, then $F(\mu_F)\geq G(\mu_G)$.
	\end{enumerate}
\end{corollary}
The interpretation of this corollary helps clarify the rationale behind the convex transform order and its relation to skewness and tail-heaviness. The dominating random variable is more likely to take values that are less than the expected.	This means, for instance, that if $F\in\mathcal{F}_U^{cx}$, $P(X\leq \E X)\leq 1/2$; if $F\in\mathcal{F}_\mathcal{E}^{cx}$, $P(X\leq \E X)\leq 1-1/e$; if $F\in\mathcal{F}_\mathcal{L}^{cx}$, we only have the trivial bound $P(X\leq \E X)\leq 1$ because $\mu_\mathcal{L}=+\infty$.

\section{Estimating the expected order statistics}\label{section3}

Given a random sample $X_1,...,X_n$ of size $n$ from $F$, let
$\F_n(x)=\frac1n\sum_{i=1}^n\mathbf{1}(X_i\leq x)$ be the corresponding empirical CDF. In general, we will denote a realisation of $\F_n$ with $F_n$, that is, the empirical CDF obtained from an observed sample $x_1,...,x_n$. Our testing approach leverages Proposition~\ref{prop}, therefore we need to estimate the expected order statistics $\mu_{j:m}$ based on a sample of size $n$, where $n$ and $m$ are generally different. Given that $X_{j:m}$ has CDF $F_{B_{j:m}}\circ F$, we can express the functional $\mu_{j:m}(F)$ as the integral ${\mu}_{j:m}(F)=\int_\R xdF_{B_{j:m}}\circ F(x)$. Writing $\mu_{j:m}(F)=\int_0^1 F^{-1}(p)dF_{B_{j:m}}(p)$, the expected order statistics belong to the family of L-functionals \citep{serfling}. By the plugin method, we propose the following estimator:
$$\hat{\mu}_{j:m}: =\mu_{j:m}(\F_n)=\int_\R xdF_{B_{j:m}}\circ \F_n(x)=\sum_{i=1}^n X_{i:n}(F_{B_{j:m}}(\tfrac in)-F_{B_{j:m}}(\tfrac{i-1}n)).$$
The statistical functional $\hat{\mu}_{j:m}$ is a weighted average, in which order statistics $X_{i:n}$ from $F$ are scaled by the probability that an order statistic $U_{j:m}$ from the uniform distribution $U$ belongs to $((i-1)/n,i/n]$,
$$\hat{\mu}_{j:m} =\sum_{i=1}^n X_{i:n}P(U_{j:m}\in(\tfrac{i-1}n,\tfrac{i}n]).$$
The estimators $\hat{\mu}_{j:m}$ are L-estimators, namely, linear combinations of order statistics. The following proposition addresses some basic finite sample properties of $\hat{\mu}_{j:m}$. By linearity, the average of the $\hat{\mu}_{j:m}$'s coincides with the sample mean $\hat{\mu}=\hat{\mu}_{1:1}=\sum_{i=1}^nX_i/n$. This is the empirical counterpart of the basic fact that $\frac1m\sum_{j=1}^m{{\mu}}_{j:m}={\mu}$. Moreover, $\hat{\mu}_{j:m}$ is increasing in $j$ and decreasing in $m$.
\begin{proposition}\label{increasing}Given a random sample from $F$, the following relations between random variables hold surely.
	\begin{enumerate}
		\item $\frac1m\sum_{j=1}^m{\hat{\mu}}_{j:m}=\hat{\mu}$;
		\item $\hat{\mu}_{j:m}\leq\hat{\mu}_{j+1:m}$;
		\item $\hat{\mu}_{j:m+1}\leq\hat{\mu}_{j:m}$.
	\end{enumerate}
\end{proposition}
By Theorem 1.A.1 in \cite{shaked}, the above results hold stochastically, namely, with $=_d$ instead of $=$ and $\leq_{st}$ instead of $\leq$, if the samples from $F$ are allowed to differ, yielding CDFs $\F_n$ and $\F'_n$.

We now focus on the asymptotic properties of these estimators, letting $n\to\infty$ while holding $j$ and $m$ fixed. It is easy to see that the density $f_{B_{j:m}}$ is always bounded by $m$. This condition, plus the integrability of $X$, is sufficient to establish a.s. convergence of $\hat{\mu}_{j:m}$ to $\mu_{j:m}$, by Theorem 2.1 of \cite{zwet1980} (see also the related work of \cite{mason,wellner}). 
\begin{proposition}\label{convergence}
If $X$ has finite mean, then, for $n\to\infty,$ ${\mu}_{j:m}(\F_n)\to_{a.s.}\mu_{j:m}(F)$, for every fixed $j$ and $m$.
\end{proposition}
However, the existence of a finite mean can be a limitation, because, as we discussed earlier, this condition is not strictly necessary for having finite expected order statistics. In the remainder of this section, we show that ${\mu}_{j:m}(\F_n)\to_{a.s.}\mu_{j:m}(F)$ can hold under weaker assumptions. These assumptions rely on some notions that are well-known in extreme value theory \citep{haan,resnick}. In particular, the following definition describes a wide family of heavy-tailed distributions that we shall deal with. Let $\Phi_\alpha(x)=\exp\{-x^{-\alpha}\}$, $x\geq0, \alpha>0,$ be the CDF of the Fréchet distribution.
\begin{definition}
We say that $F$ is in the maximum domain of attraction of $\Phi_\alpha$, with tail parameter $\alpha,$ and write $F\in \mathcal{D}^+(\Phi_\alpha)$, if there exist sequences $a_n$, $b_n$, such that $(X_{n:n}-b_n)/a_n$ converges in distribution to $\Phi_\alpha,$ namely $P(a_n^{-1}(X_{n:n}-b_n)\leq x)\to \Phi_\alpha(x)$. Similarly, $F$ belongs to the minimum domain of attraction of $\Phi_\beta$, denoted as $F\in\mathcal{D}^-(\Phi_\beta)$, if there exist sequences $c_n$, $d_n,$ such that $P(c_n^{-1}(X_{1:n}-d_n)\leq x)\to \Phi_\beta(x)$. 
\end{definition}
The above properties, $F\in\mathcal{D}^+(\Phi_\alpha)$ and $F\in\mathcal{D}^-(\Phi_\beta)$, are related to the asymptotic behaviour of the right and left tails, which are determined by the right tail parameter $\alpha$ and the left tail parameter $\beta$, respectively. To have $F\in\mathcal{D}^+(\Phi_\alpha)$ or $F\in\mathcal{D}^-(\Phi_\beta)$, it is necessary that $F$ has right or left-unbounded support, respectively. Clearly, for symmetric distributions, $\alpha$ and $\beta$ coincide. The properties of $\mathcal{D}^+$ and $\mathcal{D}^-$ can be analysed symmetrically, using the relation $\min\{X_1,\ldots,X_n\}=-\max\{-X_1,\ldots,-X_n\},$ therefore it is sufficient to focus on the behaviour of the right tail. A necessary and sufficient condition for $F\in \mathcal{D}^+(\Phi_\alpha)$ is that $$ \lim_{t\to\infty} \frac{1-F(tx)}{1-F(t)}=x^{-\alpha}$$ \citep[Theorem 1.2.1]{haan}.
The parameter $\alpha$ determines the weight of the tail, with smaller values corresponding to heavier tails. In particular, for a distribution in $\mathcal{D}^+(\Phi_\alpha)$, the moments of order greater than or equal to $\alpha$ do not exist. Most heavy-tailed models are in $\mathcal{D}^+(\Phi_\alpha),$ for example, the Fréchet (which indeed is \textit{max-stable}), the Pareto, the Burr type III and type VII, the F, the beta type II, the log-gamma, the inverse gamma, the loglogistic, the stable distribution (with shape parameter less than 2), and the Student's $t$, which includes the Cauchy. For instance, it is easy to see that the log-logistic distribution, with CDF $\mathcal{L}(x^a)$ and shape parameter $a>0,$ belongs to $\mathcal{D}^+(\Phi_{a})$, i.e., $\alpha=a,$ in particular, $\mathcal{L}\in\mathcal{D}^+(\Phi_1)$. The Cauchy distribution belongs to $\mathcal{D}^+(\Phi_1)\cap \mathcal{D}^-(\Phi_1)$.

The following result establishes a strong law of large numbers for $\hat{\mu}_{j:m}$ when $F$ does not have a finite mean.
\begin{theorem}\label{convergence2}Assume that $\E|X|=\infty.$ If either of the following conditions holds, then, for $n\to\infty,$ ${\mu}_{j:m}(\F_n)\to_{a.s.}\mu_{j:m}(F)$.
\begin{enumerate}
	\item $F\in\mathcal{D}^+(\Phi_\alpha)$, $\E (X_-)<\infty$, and $j<m+1-1/\alpha$;
	\item $F\in\mathcal{D}^-(\Phi_\beta)$, $\E (X_+)<\infty$, and $j>1/\beta$;
	\item $F\in\mathcal{D}^+(\Phi_\alpha)\cap\mathcal{D}^-(\Phi_\beta)$ and $1/\beta<j<m+1-1/\alpha$.
\end{enumerate}

\end{theorem}


Theorem~\ref{convergence2} can be easily applied in many relevant cases. The conditions depend on the relation between the tail parameter and the ranks of the order statistics considered. For example, in case 1, when $\alpha=1$ (as in the case where $F=\mathcal{L}$) we can establish a.s. convergence of $\hat{\mu}_{j:m}$ for $j=1, \ldots, m-1$. If $\alpha$ decreases, we have convergence in a smaller set of order statistics. Differently, if $\alpha>1$, the mean is finite and we can rely on Proposition~\ref{convergence}, implying convergence for all values of $j$. The basic and intuitive rule is that, when $\E X_-$ is finite, we can ensure the convergence letting the ratio between $j$ and $m$ be sufficiently small. Everything is reversed if the $\E X_+$ is finite. For symmetric distributions on the real line, one should choose $j$ to be close enough to $m/2$ (for $m$ even) or $(m+1)/2$ (for $m$ odd). For example, for the Cauchy distribution, we can ensure that $\hat{\mu}_{2:3}\to_{a.s.}\mu_{2:3}$.


\section{A new class of tests}\label{section4}

We now introduce a new family of tests based on the empirical verification of Proposition~\ref{prop}. Under $\mathcal{H}_{1+}^{G}$, we can expect the empirical counterparts of the differences $\pi_{j:m}^G-F(\mu_{j:m})$, $j=1,\ldots,m,$ to be positive. Our family of statistics is based on this idea. Let
$$\mathbf{v}^G_m=(\pi_{1:m}^G,...,\pi_{m:m}^G), $$
and define the random vector
$$\hat{\mathbf{V}}_m=(\widetilde{\F}_n(\hat{\mu}_{1:m}),...,\widetilde{\F}_n(\hat{\mu}_{m:m})),$$
where $\widetilde{\F}_n$ is used to denote the linear interpolator of the jump points of $\F_n$. The reasons why we use $\widetilde{\F}_n$ instead of $\F_n$ are technical and will be clarified later. Taking into account Proposition~\ref{prop}, a test statistic for $\mathcal{H}_{1+}^{G}$ is given by 
$$T_{m,p}^{G+}(\F_n)= ||(\mathbf{v}^G_m-\hat{\mathbf{V}}_m)_+||_p=\bigg(\sum_{k=1}^m(\pi_{k:m}^G-\widetilde{\F}_n(\hat{\mu}_{k:m}))_+^p\bigg)^{1/p}.$$
Symmetrically, one can test $\mathcal{H}_{1-}^{G}$ with 
$T_{m,p}^{G-}(\F_n)= ||(\mathbf{v}^G_m-\hat{\mathbf{V}}_m)_-||_p$.
The expressions of $T_{m,p}^{G+}$ and $T_{m,p}^{G-}$ define new families of test statistics, parameterised by the number of order statistics involved, $m$, and by the order of the $L^p$ norm, $p$. The effect of such parameters on the tests' performance will be addressed by simulations, while in this section we focus on more general theoretical properties, which hold regardless of the choices of $m$ and $p$.
Proofs and arguments will focus only on $T_{m,p}^{G+}$, because the properties of $T_{m,p}^{G-}$ are similar. 

Denote with $\G_n$ the empirical CDF obtained by sampling from $G$. Given some significance level $\alpha\in(0,1),$ the null hypothesis is rejected when $T_{m,p}^{G+}(F_n)\geq c_{n,m,p,\alpha}^{G^+}$, where the threshold value $c_{n,m,p,\alpha}^{G^+}$ is the $(1-\alpha)$-quantile of $T_{m,p}^{G+}(\G_n)$. The $p$-value is $P(T_{m,p}^{G+}(\G_n)\geq T_{m,p}^{G+}(F_n))$. To simplify notations, hereafter we write $c_{n,m,p,\alpha}^{G^+}=c^+_{n,\alpha}$, and denote the critical value of the test $T_{m,p}^{G-}(\G_n)$ with $c^-_{n,\alpha}$. Critical values and $p$-values can be determined by Monte Carlo methods.
\subsection{Finite sample properties}
Since the convex transform order is location and scale-invariant, we expect our test statistics to have the same property. This can be easily verified.
\begin{proposition}\label{invariance}
The family of test statistics $T_{m,p}^{G+}(\F_n)$ and $T_{m,p}^{G-}(\F_n)$ are location and scale-invariant.
\end{proposition}

Most theoretical properties of the proposed family of tests are based on the following stochastic monotonicity property.
\begin{lemma}\label{lemma}
Denote by $\mathbb{H}_n$ be the empirical CDF of a random sample from $H$.
If $F\leq_c H$, then, for every positive integer $m$ and for $j=1,...,m$, $$\widetilde{\F}_n(\mu_{j:m}(\F_n))\leq_{st}\widetilde{\mathbb{H}}_n(\mu_{j:m}(\mathbb{H}_n)).$$
\end{lemma}	
Lemma~\ref{lemma} implies that the power of our class of tests is monotone with respect to the convex order. This means that, given a pair of CDFs $F$ and $H$ such that $F\leq_c H$, the probability of rejecting $\mathcal{H}_0^{G}:F\in \mathcal{G}$ in favour of $\mathcal{H}_{1+}^{G}$ ($\mathcal{H}_{1-}^{G}$) under $F$ is larger (smaller) than the probability of rejection under $H$.
\begin{theorem}\label{size}
If $F\leq_c H$, then 
\begin{enumerate}
	\item 	$P(T_{m,p}^{G+}(\F_n)\geq c_{\alpha,n}^{+})\geq P(T_{m,p}^{G+}(\mathbb{H}_n)\geq c_{\alpha,n}^{+}),$
	\item $P(T_{m,p}^{G-}(\F_n)\geq c_{\alpha,n}^{-})\geq P(T_{m,p}^{G-}(\mathbb{H}_n)\geq c_{\alpha,n}^{-}).$
\end{enumerate}
\end{theorem}
Bearing in mind that $P(T_{m,p}^{G+}(\mathbb{G}_n)\geq c_{\alpha,n}^{+})=\alpha,$ the above result implies that the tests are unbiased. Moreover, the size of the tests is always bounded by $\alpha$ for every $F\geq_c G$. This is summarised as follows.
\begin{corollary}\
\begin{enumerate}
	\item 	Under $\mathcal{H}^G_{1+}$, $P(T_{m,p}^{G+}(\F_n)\geq c_{\alpha,n}^{+})\geq \alpha.$ If $F\geq_c G,$ $P(T_{m,p}^{G+}(\F_n)\geq c_{\alpha,n}^{+})\leq \alpha,$ with equality under $\mathcal{H}^G_{0}.$
	\item Under $\mathcal{H}^G_{1-}$, $P(T_{m,p}^{G-}(\F_n)\geq c_{\alpha,n}^{-})\geq\alpha.$ If $F\leq_c G,$ $P(T_{m,p}^{G-}(\F_n)\geq c_{\alpha,n}^{-}) \leq\alpha,$ with equality under $\mathcal{H}^G_{0}.$
\end{enumerate}
\end{corollary}
The motivation for using $\widetilde{\F}_n$ instead of $\F_n$ in the construction of $\hat{\mathbf{V}}_m$ is twofold. First, it enables the derivation of Lemma~\ref{lemma} using convexity, which is not possible using step functions. Moreover, differently from $\F_n(\hat{\mu}_{j:m})$, the random variables $\widetilde{\F}_n(\hat{\mu}_{j:m})$ are continuous, yielding a continuous test statistic. This facilitates the application of the test, which does not need any kind of randomisation, since $P(T^{G+}_{m,p}(\F_n)=c_{\alpha,n})=0$ and $P(T^{G+}_{m,p}(\G_n)\leq c_{\alpha,n})=\alpha$.
\subsection{Asymptotic properties}
The above properties hold regardless of the sample size and they do not need any distributional assumption, except for continuity. The asymptotic behaviour of the proposed family of tests is addressed by the following proposition, which assumes finite expectations and relies on Proposition~\ref{consistency}. Under these conditions, we can establish the consistency of the tests, namely, for $n\to\infty,$ the probability of rejecting the null hypothesis tends to 1 when the alternative is true. This result will be generalised later to the infinite mean case, leveraging Theorem~\ref{convergence2}.


\begin{proposition}\label{consistency}Assume that $F$ and $G$ have finite expectations.\begin{enumerate}
\item 	Under $\mathcal{H}_{1+}^{G}$, $P(T_{m,p}^{G+}(\F_n)\geq c_{\alpha,n}^{+})\to1$.
\item 	Under $\mathcal{H}_{1-}^{G}$, $P(T_{m,p}^{G-}(\F_n)\geq c_{\alpha,n}^{-})\to1$.
\end{enumerate}

\end{proposition}
If we are dealing with heavy-tailed models, the assumption of Proposition~\ref{consistency} can be restrictive. However, it can be relaxed. If $F$ or $G$ do not have finite means, we can rely on Theorem~\ref{convergence2} and obtain a consistent test just by discarding those values of $j$ such that $\hat{\mu}_{j:m}$ does not converge. Denote with $j_k$, $k=1,\ldots,\ell\leq m$, those values of $j$ such that $\hat{\mu}_{j_k:m}\to_{a.s.}\mu_{j_k:m}$, or, actually, one can choose any arbitrarily small subset of $\{1,\ldots,m\}$ in which we may have convergence of all the L-estimators. This can always be checked for $G$, but not for $F$. However, if $F$ can be strongly heavy-tailed, given $m$, one can properly choose $\ell$ to ensure convergence. The choice of the values $j_k$ depends on the support and the tail parameters, as discussed in Theorem~\ref{convergence2}.

Now, we can define a constrained version of the test statistic, in which the differences $\pi_{j:m}^G-\widetilde{\F}_n(\hat{\mu}_{j:m})$ are computed just at these $\ell$ values, defined by
$$T_{m,\ell,p}^{G+}(\F_n)=\bigg(\sum_{k=1}^\ell(\pi_{j_k:m}^G-\widetilde{\F}_n(\hat{\mu}_{j_k:m}))_+^p\bigg)^{1/p}.$$
The complementary test statistic $T_{m,\ell,p}^{G-}(\F_n)$ can be obtained similarly. We can now obtain critical values, denoted again as $c_{\alpha,n}^{+}$ and $c_{\alpha,n}^{-}$, and follow the same procedures described above for the unconstrained tests. It is easy to check that this new family of tests satisfy all the finite sample size properties of $T_{m,p}^{G+}(\F_n)$ and $T_{m,p}^{G-}(\F_n)$, namely, monotone power and unbiasedness. Moreover, now, even if $F$ and $G$ do not have finite means, we can establish consistency of the tests just by relying on Theorem~\ref{convergence2}. This is summarised in the following proposition, which can be proved using the same arguments of Proposition~\ref{consistency}.
\begin{proposition}\label{consistency3}Assume that $F$ and $G$ satisfy the assumptions of Theorem~\ref{convergence2} for $j=j_1,\ldots,j_\ell$.\begin{enumerate}
\item 	Under $\mathcal{H}_{1+}^{G}$, $P(T_{m,\ell,p}^{G+}(\F_n)\geq c_{\alpha,n}^{+})\to1$.
\item 	Under $\mathcal{H}_{1-}^{G}$, $P(T_{m,\ell,p}^{G-}(\F_n)\geq c_{\alpha,n}^{-})\to1$.
\end{enumerate}

\end{proposition}



\section{Simulations}\label{section5}
In this section, we conduct a numerical investigation across various alternatives, to evaluate the performance of the tests and to assess the impact of the parameters \( m \) and \( p \) on the tests' effectiveness. All the simulations are conducted using \( 5000 \) Monte Carlo trials, and the power of the test is computed using a significance level $\alpha=0.1$. The R code is available at {https://github.com/MohammedEssalih/Project-3}. We focus on the following choices of $G$: exponential (\( G = \mathcal{E} \)), log-logistic (\( G = \mathcal{L} \)), and negative exponential (\( G = \mathcal{E}_{-} \)), namely, we test the IHR/DHR, IOR/DOR, and DRHR/IRHR properties, respectively. For the IOR/DHR and DRHR/IRHR cases, it seems that there are no existing tests available for comparison (\cite{lando2024} introduced a test that has the IOR property as the null hypothesis, but this is a different testing problem). For the IHR case, we compare our tests to the well-known test first introduced by \cite{proschan1967} (abbreviated as P\&P), but also studied, among others, by \cite{bickel} and \cite{gibels}. This comparison is restricted to \( G = \mathcal{E} \), as the P\&P approach is designed exclusively for testing IHR/DHR alternatives. 

We will mainly simulate from the following models: Weibull distribution with shape parameter $a$ and scale parameter $b$, represented as $W(a,b)$; log-logistic distribution with shape parameter $a$ and scale parameter $b$, represented as $\mathcal{L}(a,b)$. It is easy to see that these are monotone in $a$ with respect to the convex ordered, that is, for $a>a'\geq0$ and for every $b,b'>0$, $W(a',b')\leq_c W(a,b)$ and $\mathcal{L}(a',b')\leq_c \mathcal{L}(a,b)$ (the convex transform order is scale-invariant, so hereafter scale parameters will be set to 1).

Simulations are performed for various values of \( m \), across different sample sizes \( n = 25, 50, 100, 200 \). The effect of the choice of $m$ on the power of the tests in some special cases is analysed in detail in the last subsection. In the following analyses, for the IHR/DHR and DRHR/IRHR cases, we take $m=1,5,10,20$. For the IOR case, coherently with the conditions of Theorem~\ref{convergence2}, we use the constrained test statistics $T_{m,\ell,p}^{G+}$ with $\ell=m-2$ ($\ell\leq m-1$ ensures convergence); thus, we use $m=3,5,10,20$. The DOR case is the most critical one since $F$ can be strongly heavy-tailed. Then, we need to choose $m$ and $\ell$ properly, to ensure that the conditions of Theorem~\ref{convergence2} hold, as we will discuss below. 

\subsection{Optimal choice of $p$}
Simulations reveal that the test statistic based on the $L^1$ norm often outperforms those derived from other norms by more swiftly detecting departures from the null hypothesis.  In the case, $ G = \mathcal{E}$, Table~\ref{Tab1} presents the rejection rates for different values of $p$ and $m$, obtained from a $W(1.5,1)$, which is a strictly IHR model. The optimal choice of $p$ may depend on $m$ and $n$, however, especially when the sample size is small, the largest power is often obtained for $p=1.$ For this reason, hereafter we will focus just on the case $p=1$.

	\begin{table}[h]
		\small
		\centering
		\setlength{\tabcolsep}{6pt} 
		\renewcommand{\arraystretch}{1.2} 
		
		\caption{Rejection rates for the IHR case, where $F\sim W(1.5,1)$}
		\label{Tab1}
		\begin{tabular}{cccccc}
			\toprule
			$p$ & $m$ & $n=25$ & $n=50$ & $n=100$ & $n=200$ \\
			\midrule
			\multirow{4}{*}{1} 
			& 1  & 0.3284 & 0.4982 & 0.7210 & 0.9002 \\
			& 5  & 0.3876 & 0.5508 & 0.8238 & 0.9722 \\
			& 10 & 0.3318 & 0.5186 & 0.7908 & 0.9688 \\
			& 20 & 0.2380 & 0.4362 & 0.7124 & 0.9440 \\
			\midrule
			\multirow{4}{*}{2}  
			& 1  & 0.2340 & 0.4688 & 0.7058 & 0.9070 \\
			& 5  & 0.2964 & 0.4676 & 0.7006 & 0.9458 \\
			& 10 & 0.2332 & 0.3512 & 0.6204 & 0.8914 \\
			& 20 & 0.1930 & 0.2912 & 0.4736 & 0.7510 \\
			\midrule
			\multirow{4}{*}{$\infty$}  
			& 1  & 0.3250 & 0.4804 & 0.7174 & 0.9124 \\
			& 5  & 0.2352 & 0.3465 & 0.5164 & 0.7596 \\
			& 10 & 0.1956 & 0.2702 & 0.3544 & 0.5426 \\
			& 20 & 0.1598 & 0.2006 & 0.2494 & 0.3406 \\
			\bottomrule
		\end{tabular}
	\end{table}

\subsection{DRHR and IRHR cases}

The alternatives are generated from the negative Weibull distribution \( X = -Y \), where \( Y \sim W(a, 1) \), which is IRHR for \( a < 1 \) and a DRHR for \( a \geq 1 \). This is an important model in extreme value theory \citep{haan}. We report only the rejection rates for the DRHR alternatives, as both IRHR and DRHR cases exhibit the same pattern. The results, reported in Figure~\ref{IRHR}, show that the power function is monotone with respect to \( a \) and increases for every fixed \( a \) as \( n \) grows, demonstrating the monotonicity and consistency properties of the tests. 

\begin{figure}[h]
	\centering
	\includegraphics[scale=0.8]{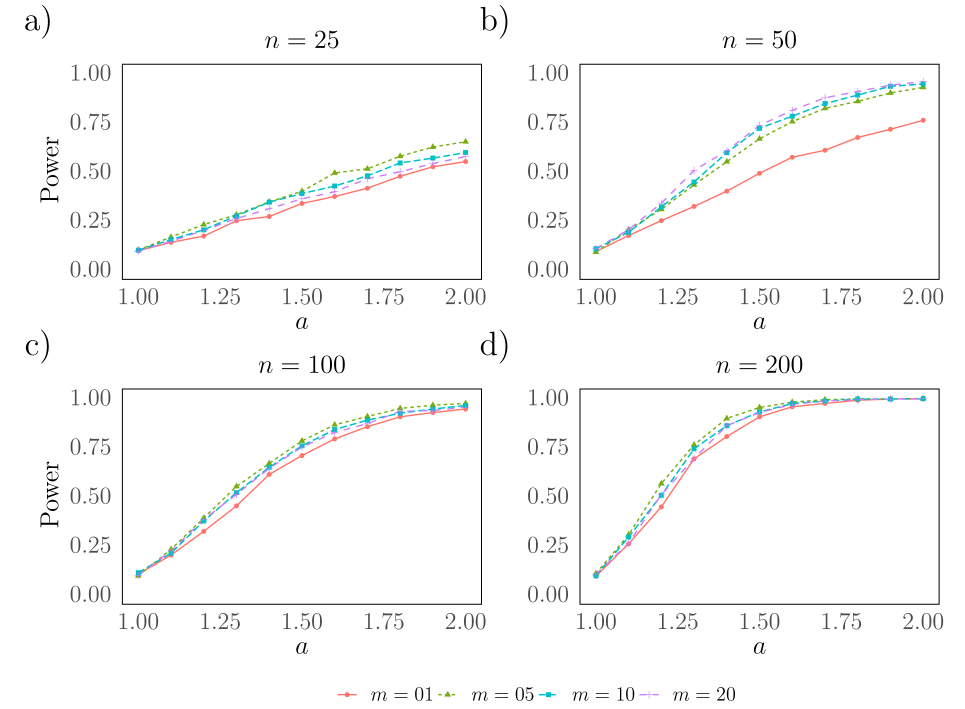}
	\caption{Simulated power of the proposed DRHR tests under the negative \( W(a, 1) \), where \(a\in[1, 2]\).}
	\label{IRHR}
	
	\centering
	\includegraphics[scale=0.8]{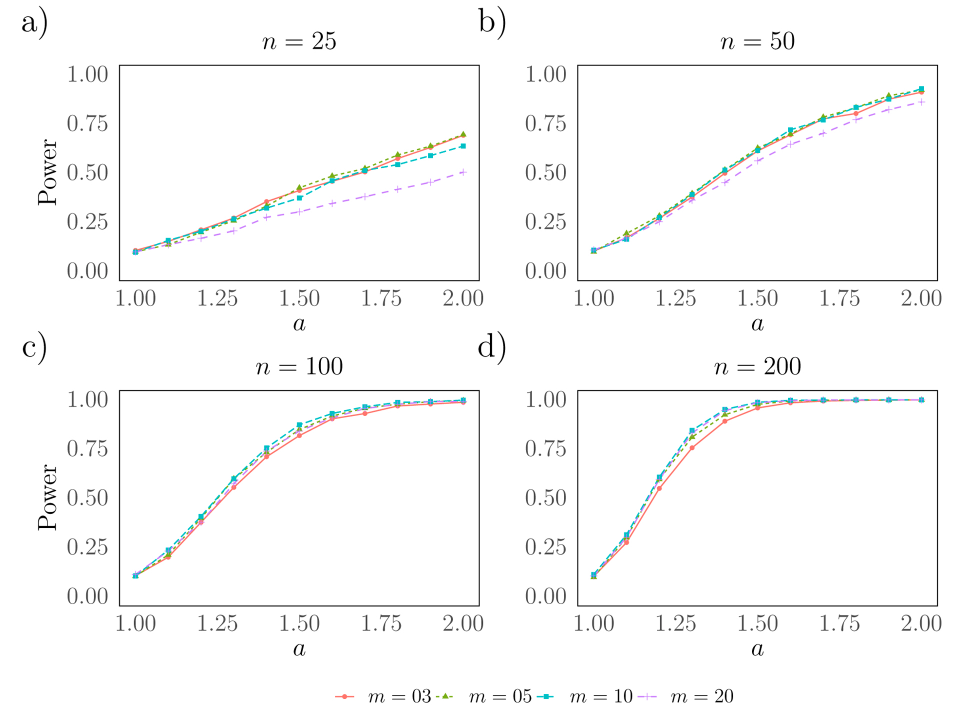}
	\caption{Simulated power of the proposed IOR tests under \( \mathcal{L}(a, 1) \), where \(a\in[1, 2]\).}
	\label{IOR}
\end{figure}

\subsection{IOR and DOR cases}
The IOR and DOR families have important mathematical properties that have garnered significant attention in recent studies. This motivates the interest in nonparametric tests for their verification. For instance, the IOR family has applications in reliability and survival analysis, as it extends the concept of negative ageing to distributions that do not necessarily have all moments \citep{odds}. The DOR family, also referred to as super-Pareto \citep{chen2024}, exhibits the following property: if $X\sim F$ is DOR, a convex combination of an iid sample from $X$ (for example, the sample mean) is stochastically larger than $X$. This result, recently proved by \cite{chen2024}, has implications in the fields of insurance and finance, as it means that, somewhat surprisingly, risk sharing can have a negative effect on participants under a DOR distributional assumption.

We conduct simulations using the log-logistic distribution \( \mathcal{L}(a, 1) \), which is IOR for \( a \geq 1 \) and DOR for $a\leq 1$. For the IOR test, we compute $T_{m,\ell,p}^{G+}$ with $\ell=m-2$, in order to guarantee consistency (any $\ell\leq m-1$ would be enough). The plots in Figure~\ref{IOR} summarise the output of this implementation, highlighting the consistency of the tests. For the DOR case, we let $a\in[0.1,1]$, hence we know that the tail parameter $\alpha$ is greater than or equal to $0.1$. This means that we must choose $j\leq m-10$ to ensure consistency. Hence, we take $m=25$ and $\ell=5$; $m=30$ and $\ell=10$; $m=35$ and $\ell=15$; $m=40$ and $\ell=20$. The results, reported in Figure~\ref{DOR}, show consistency and monotonicity of the tests; the best performance is obtained for $\ell=15$ and $\ell=20$. In general, $\alpha$ is unknown. In this case, a first ``safe" option is to fix a lower bound for $\alpha$, as in the current example, and choose $m$, $\ell$ accordingly. Otherwise, one may estimate $\alpha$, for example using the Hill estimator \citep{haan}.

\subsection{IHR and DHR cases: a comparison with P\&P}

We proceed further by comparing the present tests to that of P\&P, to detect IHR and DHR properties. The test of P\&P is based on the normalised spacings $\overline{D}_{i,n}=(n-i+1)(X_{n-i+1:n}-X_{n-i:n})$. An increasing behaviour of the HR is related to a stochastically decreasing behaviour of the random variables $\overline{D}_{i,n}$. Then, the test statistic is given by the number of times that $\overline{D}_{i,n}\geq \overline{D}_{j,n}$, for $i < j$, that is, $V_n=\sum_{i<j}V_{i,j}$, where $V_{i,j}=\mathbf{1}(\overline{D}_{i,n}-\overline{D}_{j,n}>0)$. The null hypothesis is rejected when the value of $V_n$ is larger than its $(1-\alpha)$ quantile obtained under exponentiality. The test for DHR is easily obtained by reversing signs. 

In the IHR case, we implement both tests and plot the rejection rates for various values of \( m \) across multiple values of \( n \) using the Weibull distribution \( W(a,1) \), for \( a \in [1, 2] \). Note that the case $a=1$ is the exponential, which yields a rejection probability equal to $\alpha$. The obtained results are summarised in Figure~\ref{LandoVsP&P}. The behaviour in the DHR case is similar. The results show that when the distribution is IHR or DHR, the tests of P\&P deliver larger power compared to ours, regardless of $p$ and $m$, and this is especially apparent for small sample sizes. However, our tests have another advantage. Indeed, tests for exponentiality versus monotone HR alternatives can be misleading when the underlying distribution has a non-monotone hazard rate. In such cases, our families of tests, especially with larger values of $m$, can be seen to be overall more robust than the P\&P tests. We simulate from the Student's \( t \)-distribution with 1.1 degrees of freedom, denoted as \( St(1.1) \), which has a bell-shaped HR. The rejection rates for testing the IHR and DHR alternatives are reported in Table~\ref{LandoVsP&P_T}. On the one hand, Table~\ref{LandoVsP&P_T}-(a) shows that the two P\&P tests incorrectly suggest that the $St(1.1)$ distribution is IHR, especially for large sample sizes, without detecting the DHR behaviour. On the other hand, Table~\ref{LandoVsP&P_T}-(b) reveals that our tests lead to rejecting exponentiality in favour of both IHR and DHR alternatives. Since these alternatives are contradictory, one can infer that the distribution being tested is IHR in some interval, but DHR in some other interval, as is the case for $St(1.1)$. Similar results, not reported here, hold simulating from the $St(1)$ (which coincides with the Cauchy, and has no mean), $St(1.5)$ and $St(2)$; and for beta distributions with parameters less than 1, which have a decreasing and then increasing HR. The misidentification highlights a limitation of the P\&P test when assessing distributions with non-monotone hazard rates, while our tests, especially when $m$ is not too small, are overall more reliable in such cases. 

We also computed the test proposed by \cite{mitra2008}, which is based on L-estimators as well (see also \cite{anis} for a generalisation of this approach). This test delivers large power under Weibull-distributed IHR alternatives. However, its applicability is limited only to nonnegative random variables, differently from the tests considered above. Moreover, we noted that it is not location invariant, which is a problem. For instance, if $F(x)=\mathcal{E}(x-1),x>1,$ namely, a shifted exponential, for $n=50$ we obtained a simulated power equal to 1, instead of $\alpha=0.1$, as it should be (the shifted exponential has a constant HR, just like the classic exponential, so we are under $\mathcal{H}_0^\mathcal{E}$ in this case). For these reasons, we did not include this test in the comparison.

\begin{table}[h]
	\centering
	\small
	\setlength{\tabcolsep}{6pt}  
	\renewcommand{\arraystretch}{1.2}  
	
	\caption{Rejection rates for IHR and DHR tests, where the latter are shown in parentheses, for the $St(1.1)$.}
	\label{LandoVsP&P_T}
	
	\begin{subtable}{\textwidth}
		\centering
		\caption{P\&P test}
		\begin{tabular}{ccccc}
			\toprule
			$n=25$ & $n=50$ & $n=100$ & $n=200$ & $n=500$ \\
			\midrule
			0.49 (0.01) & 0.88 (0) & 0.98 (0) & 0.99 (0) & 1 (0) \\
			\bottomrule
		\end{tabular}
	\end{subtable}
	
	\vspace{0.5cm} 
	
	\begin{subtable}{\textwidth}
		\centering
		\caption{Proposed tests}
		\begin{tabular}{c ccccc}
			\toprule
			$m$ & $n=25$ & $n=50$ & $n=100$ & $n=200$ & $n=500$ \\
			\midrule
			1  & 0.57 (0.22) & 0.60 (0.24) & 0.62 (0.27) & 0.65 (0.26) & 0.68 (0.27) \\
			5  & 0.93 (0.15) & 0.99 (0.31) & 1 (0.58)    & 1 (0.88)    & 1 (1) \\
			10 & 0.92 (0.09) & 0.99 (0.16) & 1 (0.37)    & 1 (0.74)    & 1 (0.99) \\
			20 & 0.85 (0.04) & 0.99 (0.06) & 1 (0.13)    & 1 (0.38)     & 1 (0.94) \\
			\bottomrule
		\end{tabular}
		\label{LandoVsP&P_T}
	\end{subtable}
\end{table}

\begin{figure}[h]
	\centering
	\includegraphics[scale=0.8]{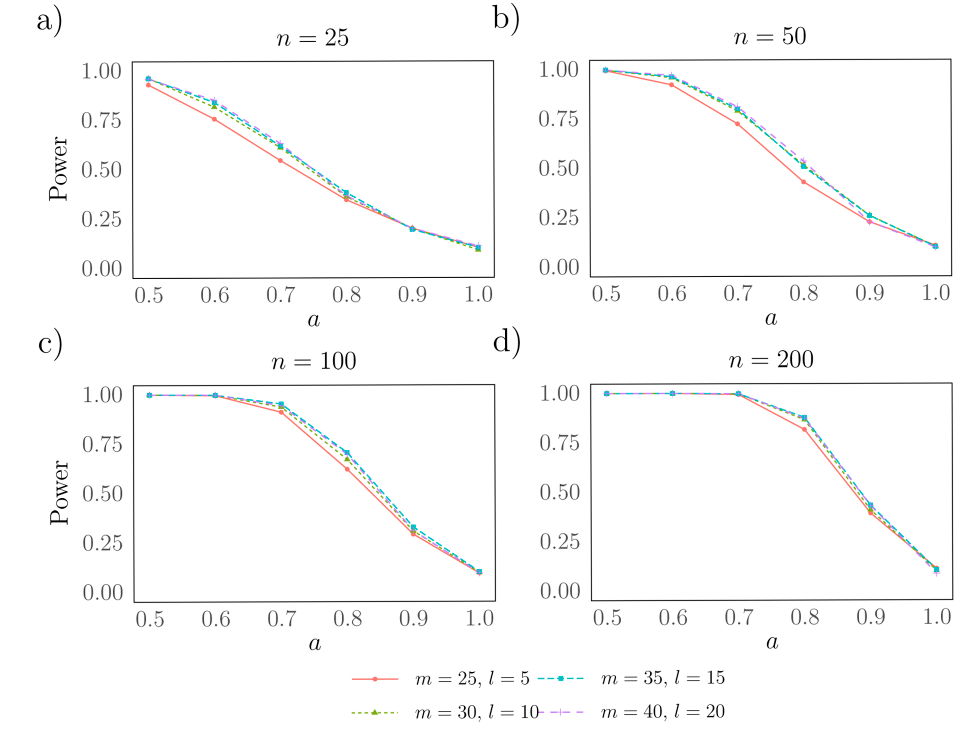}
	\caption{Simulated power of the proposed DOR tests under \( \mathcal{L}(a, 1) \), where \(a\in[0.1, 1]\). The power is always 1 for $a<0.5$, so we excluded these values from the range of the plot.}
	\label{DOR}
	\centering
	\includegraphics[scale=0.8]{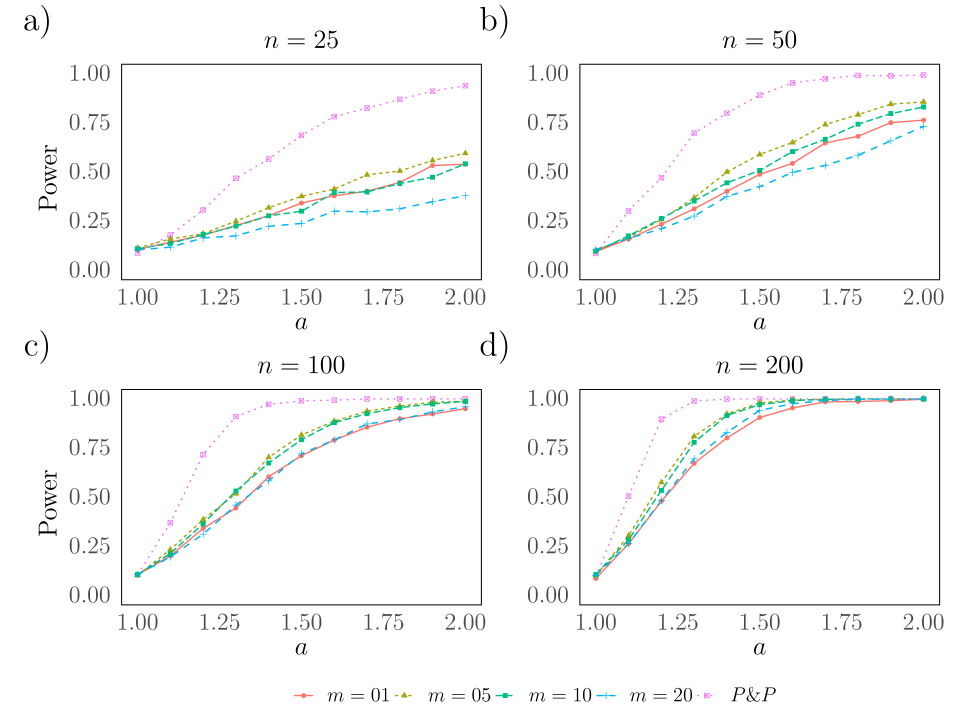}
	\caption{Simulated power of the proposed IHR tests compared to that of P\&P under \( W(a, 1) \), where \(a\in[1, 2]\).}
	\label{LandoVsP&P}
\end{figure}

\subsection{Optimal value of $m$}

This subsection aims to identify the optimal values of \( m \), which seem to affect the performance of the tests. Overall, too small values of $m$ can be risky, for example, they could lead more often to the wrong decision when the underlying distribution has a non-monotone generalized HR, as it has been shown earlier for $G=\mathcal{E}$. Therefore, we suggest larger values on $m$, for example, $m\geq 5$. Moreover, when $G$ is heavy-tailed, we need larger values of $m$ to ensure consistency. The optimal value of \( m \) also depends on the sample size. 

To analyse this relationship, we define the DRHR distribution as the negative $W(1.5, 1)$, the IOR as $\mathcal{L}(1.5, 1)$, and the IHR distribution as $W(1.5, 1)$. A 3D plot, represented in Figure~\ref{captures}, illustrates the rejection rates for various combinations of $n$ and $m$.

For all the tests, small values of $m$ result in poor test performance. For the IOR and IHR tests, the optimal $m$ is observed to lie within the interval corresponding to $10\%$ to $20\%$ of $n$. For each $n$, the power increases with $m$, peaks at a certain point, and then decreases as $m$ continues to grow. In contrast, the DRHR test demonstrates a different pattern: the power initially rises with $m$ but then stabilises and remains constant for larger values of $m$. Finally, the influence of $m$ on the power is more pronounced for smaller values of $n$. Indeed, as $n$ increases, by consistency, the power approaches $1$ regardless of $m$.

\begin{figure}[ht]
	\centering
	\includegraphics[scale=0.35]{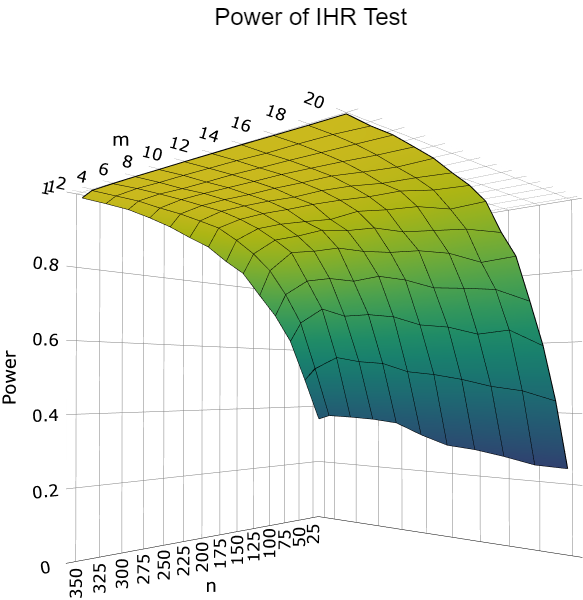}
	\includegraphics[scale=0.35]{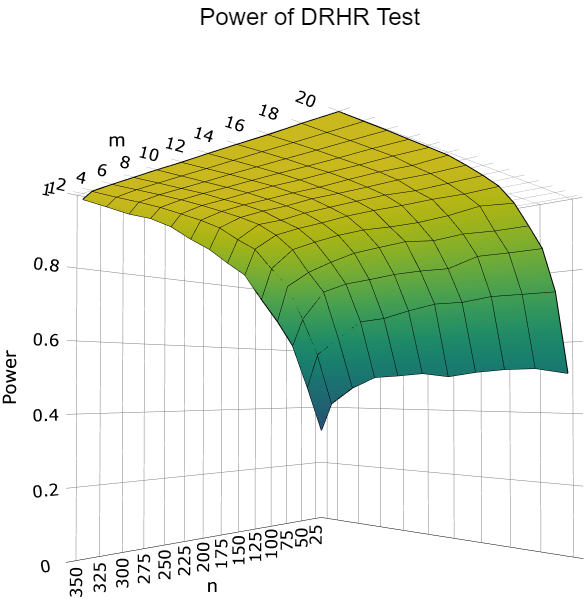}
	\includegraphics[scale=0.35]{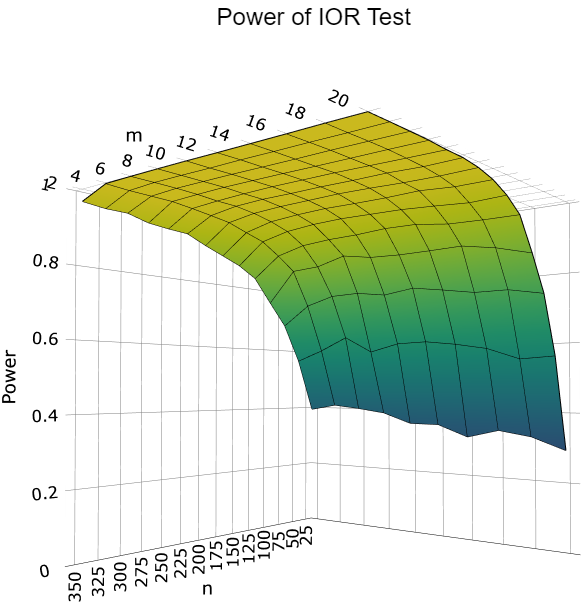}
	\caption{Simulated power of the proposed tests evaluated for different combinations of $n$ and $m$.}\label{captures}
\end{figure}

\section{A real data example}\label{section6}

In the following example, we demonstrate how our family of tests can be applied iteratively to identify the shape properties of a distribution by analysing a sample from it. As we will discuss, this approach allows testing multiple hypotheses to uncover the key property of the distribution of interest, ultimately aiding in the selection of the appropriate model to fit the data.

We apply our method to a dataset in \cite{bryson}, which reports annual flows of the Weldon River at Mill Grove, Missouri, during the years 1930-59 ($n$=26). One can wonder whether the underlying distribution is heavy-tailed, which is very important to predict the risk of extreme events, as these can cause floods. \cite{bryson} claims that the data indicates a gamma distribution. By fitting this distribution to the data, we estimate a shape parameter of $0.64$ and a scale parameter of $691$ via maximum likelihood. A Kolmogorov-Smirnov test suggests that this hypothesis could be true, with a $p$-value of 0.58. However, this distribution has all moments finite, so it is definitely not heavy-tailed. This model might be unable to predict the risk of extreme events. We then apply our tests to better understand the situation. By applying the tests for IHR/DHR, we find that the distribution is likely to be DHR. Using $m=5$ and $p=1$, the null hypothesis of exponentiality versus the strict DHR alternative is rejected with a $p$-value 0.01. This confirms the finding of \cite{bryson}, as the gamma distribution is strictly DHR, when the shape parameter is less than 1. Now, we apply our constrained tests for IOR/DOR, with $m=5$ and $\ell=3$. Both tests do not reject the null hypothesis $\mathcal{H}_0^{\mathcal{L}},$ in favour of strict IOR and DOR alternatives, with $p$-values 0.3 and 0.5, respectively. This indicates that the distribution could have a constant odds rate, namely, $F$ may be a log-logistic of the form $F(x)=\mathcal{L}(x/s),$ for some scale parameter $s>0$. A Kolmogorov-Smirnov test confirms this suggestion, with a $p$-value 0.84. More generally, one can fit the data using a log-logistic model of the form $F(x)=\mathcal{L}(a,s)(x)=\mathcal{L}((x/s)^a),$ where the shape parameter $a$ is allowed to vary. In this case, the maximum likelihood estimates are $\hat{a}=1.13$ and $\hat{s}=185$. The Kolmogorov-Smirnov test supports this finding, yielding a larger $p$-value of 0.94. This distribution is heavy-tailed, which makes a remarkable difference compared to the gamma distribution. In particular, the estimated log-logistic distribution has mean 1412 and infinite variance (precisely, all moments of order $\geq1.13$ are not finite), whereas the estimated gamma has mean 442, which coincides with the mean of the sample, and finite variance. An alternative option consists in fitting the data using the IOR shape-constrained estimator of \cite{lando2024}. In this case, the Kologorov-Smirnov test yields a $p$-value of 0.98. In this example, our tests prove useful in detecting the tail behaviour of the underlying distribution, indicating a heavy-tailed distribution, which is more suitable to model the risk of extreme events.

\begin{appendix}
	\section{Proofs}
\label{app_proofs}
\begin{proof}[Proof of Proposition~\ref{increasing}]The following results hold for every realisation $F_n$ of $\F_n$, that is, for every outcome in the sample space. Hence, we obtain sure relations.
\begin{enumerate}
\item 	Denote with $X^*$ a random variable with CDF $F_n$. By linearity, $$ \frac1m\sum_{j=1}^m \E X^*_{j:m}=\E\big(\frac1m\sum_{j=1}^m X^*_{j:m}\big)=\E X^*.$$
Now, the conclusion follows from noticing that $\E X^*_{j:m}={\mu}_{j:m}(F_n)$ and $\E X^*=\sum_{i=1}^nx_i/n.$ 
\item By definition of order statistics, $X^*_{j:m}\leq X^*_{j+1:m}$, then, taking expectations, ${\mu}_{j:m}(F_n)\leq{\mu}_{j+1:m}(F_n)$. 
\item It's easy to check that $B_{j:m+1}\leq_{st} B_{j:m}$, hence, for every realisation $F_n$, $F_{B_{j:m+1}}\circ F_n(x)\geq F_{B_{j:m}}\circ F_n(x)$, for every $x$, in other words, $F_{B_{j:m+1}}\circ F_n\leq_{st}F_{B_{j:m}}\circ F_n$. Since the mean is isotonic with the usual stochastic order, ${\mu}_{j:m+1}(F_n)\leq{\mu}_{j:m}(F_n)$.
\end{enumerate}
\end{proof}

The proof of Theorem~\ref{convergence2} is based on the following argument.	The Glivenko-Cantelli theorem and the uniform continuity of $F_{B_{j:m}}$ in $[0,1]$ imply $\sup_{[0,1]}|F_{B_{j:m}}\circ \F_n- F_{B_{j:m}}\circ F|\to_{a.s.}0.$ Since the (random) CDF $F_{B_{j:m}}\circ \F_n$ converges also pointwise a.s. to the CDF $F_{B_{j:m}}\circ F$, we have a convergence in distribution which holds with probability 1. Theorem 3.5 in \cite{billingsley} establishes that convergence in distribution implies convergence in expectation, under uniform integrability of the converging sequence. This means that, if $Z_n\to_d Z$ and the sequence $Z_n$ is uniformly integrable, that is, if $\lim_{M\to\infty}\sup_n\int_{|Z_n|>M}|Z_n|dP\to0,$ then $\E Z_n\to \E Z.$ Hence, in our random setting, a stochastic version of this condition may ensure that the expectation functional preserves the convergence, that is, $\mu(F_{B_{j:m}}\circ\F_n)\to_{a.s.}\mu(F_{B_{j:m}}\circ F).$ In particular, we will require that $\lim_{r\to0}\limsup_{n\to\infty}\big(\int_{0}^r+\int_{1-r}^1\big)|\F_n^{-1}(p)|dV(p)=_{a.s.}0.$ To verify this condition, we leverage the results of \cite{mason}. 

\begin{proof}[Proof of Theorem~\ref{convergence2}]
To simplify the notations, let $V=F_{B_{j:m}}$ and $v=V'=f_{B_{j:m}}$. Note that $\sup_{p\in(0,1)}v(p)=B\leq m$. We focus on case 3, which is the most general one, hence we assume that $F\in \mathcal{D}^+(\Phi_\alpha)\cap\mathcal{D}^-(\Phi_\beta),$ with $\alpha\leq1$ and $\beta\leq1$, since $\E|X|=\infty.$ Indeed, if $\E(X_-)<\infty$ ($\E(X_+)<\infty$), we can apply Proposition~\ref{convergence} to $X_-$ ($X_+$), and focus only on $X_+$ ($X_-$).
	
{	Consider the following decomposition:
	\begin{equation}\label{deco}
		\mu(V\circ\F_n)= \int_0^{r}\F_n^{-1}(p)dV(p)+\int_r^{1-r}\F_n^{-1}(p)dV(p)+\int_{1-r}^1\F_n^{-1}(p)dV(p).
	\end{equation}
	 We first show that the second term of equation~(\ref{deco}) converges a.s. to $\int_r^{1-r}F^{-1}(p)dV(p)$, as $n\to\infty$.
	For $r>0$, $F^{-1}$ is bounded in $[r,1-r]$ and $\F_{n}^{-1}$ converges a.s. and uniformly to $F^{-1}$ in $[r,1-r]$. Then,  $$\int_r^{1-r}|{\F}_{n}^{-1}(p)-{F}^{-1}(p)|dp\leq(1-2r) \sup_{p\in[r,1-r]}|{\F}_{n}^{-1}(p)-{F}^{-1}(p)|	\to_{a.s.}0.$$
H\"{o}lder's inequality gives
	\begin{eqnarray*}
		\bigg\lvert\int_r^{1-r}\F_n^{-1}(p)dV(p)-\int_r^{1-r}F^{-1}(p)dV(p)\bigg\rvert\leq	\int_r^{1-r}|\F_n^{-1}(p)-F^{-1}(p)|dV(p)\\
		\leq \bigg(\sup_{p\in[r,1-r)}v(p)\bigg)\int_r^{1-r}|\F_n^{-1}(p)-F^{-1}(p)|dp\\
		\leq B\int_r^{1-r}|{\F}_{n}^{-1}(p)-{F}^{-1}(p)|dp	\to_{a.s.}0.
	\end{eqnarray*}
	Hence, $ \int_r^{1-r}\F_n^{-1}(p)dV(p) \to_{a.s.}\int_r^{1-r}F^{-1}(p)dV(p)$.}
	
	{Now, we need to show that the two tail blocks in equation~(\ref{deco}) can be a.s. negligible if $r\to0.$ We first focus on the right-tail integral, $\int_{1-r}^1\F_n^{-1}(p)dV(p)$. Recall that $$v(p)=\frac1Mp^{j-1}(1-p)^{m-j}, \quad p\in[0,1],$$ where $M=1/B(j,m-j+1)$ and $B$ denotes the beta function. By assumption, $1/\alpha<m-j+1,$ where $\alpha\leq1$, so that $m>j$. Note that, as $m> j\geq1$, $(1-p)^{m-j}$ is strictly decreasing and $p^{j-1}$ is increasing in $[0,1],$ therefore, the product $v(p)$ is monotone decreasing for $p$ close to 1, say $p>1-\lambda$, for some $\lambda\in(0,1-(j-1)/(m-1))$, where $(j-1)/(m-1)$ is the mode of $v$. Hence, for every $i>(1-\lambda) n$, we have the following bounds
		$$\int_{(i-1)/n}^{i/n} v(p)dp=M\int_{(i-1)/n}^{i/n} p^{j-1}(1-p)^{m-j}dp\leq M \frac1n\bigg(1-\frac {i-1}n\bigg)^{m-j}. $$
		This implies that, for $r<\lambda,$ \begin{align*}\lim_{r\to0}\limsup_{n\to\infty}\bigg\vert\int_{1-r}^1\F_n^{-1}(p)dV(p)\bigg\vert\\
			=\lim_{r\to0}\limsup_{n\to\infty}\bigg\vert\sum_{i=n-[nr]+1}^{n}X_{i:n}\int_{(i-1)/n}^{i/n} v(p)dp\bigg\vert\\
		\leq M\lim_{r\to0}\limsup_{n\to\infty}\bigg\vert\sum_{i=n-[nr]+1}^{n}X_{i:n}\frac1n\bigg(1-\frac {i-1}n\bigg)^{m-j} \bigg\vert\\
		=M\lim_{r\to0}\limsup_{n\to\infty}\bigg\vert\sum_{k=1}^{[nr]}X_{n-k+1:n}n^{-1-(m-j)}k^{m-j} \bigg\vert.\end{align*}	
	 Now, let $$h(p)=(F^{-1}(p))_+.$$ 
	Theorem 2 of \cite{mason} implies that $$\lim_{n\to\infty}\sum_{k=1}^{n}X_{n-k+1:n}n^{-1-(m-j)}k^{m-j}=_{a.s.}\int_0^1(1-p)^{m-j}h(p)dp,$$ if and only if 
	$$E(h,j,m)=\int_0^1 h^{\frac{1}{m-j+1}}(p)dp<\infty, $$  and $$Q(h,j,m)=\int_0^1(1-p)^{m-j}h(p)dp<\infty. $$
	 The condition $E(h,j,m)<\infty$ is verified for $1/\alpha<m-j+1$ because in this case $F$ has all moments of order $\tau\in(0,\alpha)$ \citep[p. 9]{haan}, therefore $(F^{-1})^\tau_+$ is surely integrable. With regard to $Q(h,j,m)$, note that, if $F\in \mathcal{D}^+(\Phi_\alpha)$, then we can write
	 $1-F(x)=x^{-\alpha}L(x) $, where $L$ is slowly varying at infinity \citep[p. 15]{resnick}. It is well-known that, in this case, the tail-quantile function $U(t)=F^{-1}(1-1/t)$ is regularly varying, for $t\to\infty,$ with tail index $1/\alpha$ \citep[Theorem 1.5.12]{bingham}. This means that we can write $U(t)=t^{1/\alpha}\widetilde{L}(t)$, where $\widetilde{L}$ is slowly-varying, and $$F^{-1}(p)=(1-p)^{-1/\alpha}\widetilde{L}(1/(1-p)).$$ Now, the integral $Q(h,j,m)$ converges if and only if, for some $\gamma\in(0,1)$,
		$$ \int_{1-\gamma}^1(1-p)^{-1/\alpha+m-j}\widetilde{L}\big(\tfrac1{1-p}\big)dp=
		\int_{1/\gamma}^\infty x^{-1/\alpha+m-j-2}\widetilde{L}(x)dx<\infty , $$
	which holds again for $1/\alpha<m-j+1$ \citep[Proposition 1.5.10]{bingham}. We conlude that $\int_{1-r}^1\F_n^{-1}(p)dV(p)$ tends a.s. to a finite quantity, therefore, by letting $r\to0,$ the third term in equation~(\ref{deco}) vanishes with probability 1.}
	
	{By similar reasoning, for sufficiently small $r,$ \begin{align*}\lim_{r\to0}\limsup_{n\to\infty}\bigg\vert\int_{0}^r\F_n^{-1}(p)dV(p)\bigg\vert\\
			\leq M\lim_{r\to0}\limsup_{n\to\infty}\sum_{k=1}^{[nr]}\vert X_{k:n}\vert k^{j - 1} n^{-j} .\end{align*} 
			By Theorem 1 of \cite{mason}, and using symmetric arguments, one can show that, for $r\to0,$ the first integral in equation~(\ref{deco}) converges to 0 a.s. whenever $j>1/\beta.$ Finally, for every $\epsilon>0$, there exist some $r_\epsilon$ such that, for $r<r_\epsilon$, there exists an index $N=N(r_\epsilon)$ such that, for all $n>N$, we have a.s.,
			\begin{eqnarray*}
				|\mu_{j:m}(\F_n)-\mu_{j:m}(F)|\\
				\leq \big(\int_{0}^r+\int_{1-r}^1\big)|\F_n^{-1}(p)|dV(p)+\int_r^{1-r}|\F_n^{-1}(p)-F^{-1}(p)|dV(p)<\epsilon,
			\end{eqnarray*}concluding the proof.}
\end{proof}

\begin{proof}[Proof of Lemma~\ref{lemma}]
Let $H_n$ be a realisation of $\mathbb{H}_n$, with jumps at points $y_1,...,y_n$. Given $H_n$, let $F_n= H_n\circ H^{-1}\circ F.$ $F_n$ is the empirical CDF with jumps at points $x_i=F^{-1}\circ H(y_i),$ $i=1,...,n,$ so it may be seen as a realisation of $\F_n$. Now, $$H_n^{-1}\circ F_n=I_n\circ H^{-1}\circ F,$$
where $I_n$ is a step function such that $I_n(y_i)=y_i$, $i=1,...,n.$ Therefore the convex function $H^{-1}\circ F$ interpolates the jump points of $H_n^{-1}\circ F_n$. Since $\widetilde{H}_n^{-1}\circ \widetilde{F}_n$ and $H_n^{-1}\circ F_n$ coincide at points $x_i$, we conclude that $\widetilde{H}_n^{-1}\circ \widetilde{F}_n$ is a convex function, interpolating $H_n^{-1}\circ F_n$. Now, let $X^*$ and $Y^*$ be the RVs with CDFs ${F}_n$ and ${H}_n$. Jensen's inequality implies that

$$\E X^*_{j:m}=\int xdF_{B_{j:m}}\circ F_n(x) \leq \widetilde{F}_n^{-1}\circ \widetilde{H}_n\bigg(\int \widetilde{H}_n^{-1}\circ \widetilde{F}_n(x)dF_{B_{j:m}}\circ F_n(x) \bigg).$$
By a change of variables, taking into account that $F_n\circ\widetilde{F}_n^{-1}\circ \widetilde{H}_n=H_n,$
$$\int \widetilde{H}_n^{-1}\circ \widetilde{F}_n(x)dF_{B_{j:m}}\circ F_n(x)=\int xdF_{B_{j:m}}\circ H_n(x)=\E Y^*_{j:m},$$
which implies that
$$\widetilde{F}_n\bigg(\int xdF_{B_{j:m}}\circ F_n(x)\bigg) \leq \widetilde{H}_n\bigg(\int xdF_{B_{j:m}}\circ H_n(x) \bigg).$$
This inequality holds for every possible realization $H_n$ of $\mathbb{H}_n$. Therefore, Theorem 1.A.1 of \cite{shaked} yields that 
$\widetilde{\F}_n({\mu}_{j:m}(\F_n)) \leq_{st} \widetilde{\mathbb{H}}_n({\mu}_{j:m}(\mathbb{H}_n)).$
\end{proof}

\begin{proof}[Proof of Theorem~\ref{size}]
We prove just part 1. Given random variables $U_j$ and $Z_j$, $j=1,\ldots,m$, such that $U_j\leq_{st}Z_j$ for every $j=1,\ldots,m,$ Theorem 1.A.3 of \cite{shaked} establishes that $\psi(U_1,\ldots,U_m)\leq_{st}\psi(Z_1,\ldots,Z_m),$
for every increasing function $\psi:\R^m\to\R$. Given some vector $(z_1,\ldots,z_m),$ the function $$\widetilde{\psi}_p(|z_1|,\ldots,|z_m|)=||(z_1,\ldots,z_m)||_p=\big(\sum_{j=1}^m |z_j|^p\big)^{1/p},\quad p\geq1,$$
is increasing. Now, let		$$\widetilde{Z}_j=\big(\pi_j^G-\widetilde{\F}_n(\mu_{j:m}(\F_n))\big)_+,$$ and $$
\widetilde{U}_j=\big(\pi_j^G-\widetilde{\mathbb{H}}_n(\mu_{j:m}(\mathbb{H}_n))\big)_+.$$	By reversing signs, Lemma~\ref{lemma} implies that $\widetilde{U}_j\leq_{st}\widetilde{Z}_j$ for every $j=1,\ldots,m.$ Now, Theorem 1.A.3 of \cite{shaked} yields
$$T_{m,p}^{G+}(\mathbb{H}_n)=\widetilde{\psi}_p(|\widetilde{U}_1|,\ldots,|\widetilde{U}_m|)\leq_{st}\widetilde{\psi}_p(|\widetilde{Z}_1|,\ldots,|\widetilde{Z}_m|)=T_{m,p}^{G+}(\F_n). $$
Therefore, by definition of the usual stochastic order,		
$$P(T_{m,p}^{G+}(\F_n)\geq t)\geq P(T_{m,p}^{G+}(\mathbb{H}_n)\geq t),$$ for every $t$. Taking $t=c_{\alpha,n}^{+}$, the proof is concluded.
\end{proof}

\begin{proof}[Proof of Proposition~\ref{consistency}]
Again, we just prove the first assertion. 
\begin{align*}|\widetilde{\G}_n({\mu}_{j:m}(\G_n))-G({\mu}_{j:m}(G))|\leq\\
	|\widetilde{\G}_n({\mu}_{j:m}(\G_n))-\widetilde{\G}_n({\mu}_{j:m}(G))|+|\widetilde{\G}_n({\mu}_{j:m}(G))-G({\mu}_{j:m}(G))|.\end{align*}
Noticing that $\sup_x|\widetilde{\G}_n(x)-\G_n(x)|\leq 1/n\to0$, the two terms converge a.s. to 0 by Proposition~\ref{convergence} and by the Glivenko-Cantelli theorem, respectively. Therefore, for every $j=1,...,m,$
$\widetilde{\G}_n({\mu}_{j:m}(\G_n))\to_{a.s.} G(\mu_{j:m}(G)).$ 
Accordingly, if $\mathcal{H}^G_0$ is true, the continuity of the $L^p$ norm and the continuous mapping theorem imply
$$T_{m,p}^{G+}(\G_n)=\bigg(\sum_{k=1}^m(\pi_{k:m}-\widetilde{\G}_n(\hat{\mu}_{k:m}))_+^p\bigg)^{1/p}\to_{a.s.}T_{m,p}^{G+}(G)=0.$$
Then, the $(1-\alpha)$-quantile of $T_{m,p}^{G+}(\G_n)$, $c^+_{\alpha,n}$, also tends to 0. Under $\mathcal{H}^{G+}_1$, $G^{-1}\circ F $ is convex and not affine, therefore it is strictly convex, and Jensen's inequality holds strictly. Then, for $j\in\{1,\ldots,m\}$, $\widetilde{\F}_n(\hat{\mu}_{j:m})\to_{a.s.}F({\mu}_{j:m})< \pi_{j:m}^G$. By continuity of the $L^p$ norm, the continuous mapping theorem implies \begin{align*}T_{m,p}^{G^+}(\F_n)=\bigg(\sum_{k=1}^m(\pi_{k:m}^G-\widetilde{\F}_n(\hat{\mu}_{k:m}))_+^p\bigg)^{1/p}\to_{a.s.}\\\bigg(\sum_{k=1}^m(\pi_{k:m}^G-{F}({\mu}_{k:m}))_+^p\bigg)^{1/p}=T_{m,p}^{G^+}(F)=d,\end{align*} where $d$ is some positive number. Given some $0<\epsilon<d,$ there exists some $n_0$ such that, for $n>n_0,$ $|T_{m,p}^{G^+}(\F_n)-T_{m,p}^{G^+}(F)|<\epsilon $, a.s., so that
$	T_{m,p}^{G^+}(\F_n)>d-\epsilon$ with probability 1. However, as $c_{\alpha,n}^+\to0$ and $\epsilon$ is arbitrarily small, it follows that $\lim_{n\to\infty}P(	T_{m,p}^{G^+}(\F_n)\geq c_{\alpha,n}^+)= 1$.


\end{proof}

\end{appendix}
\bibliographystyle{elsarticle-harv}
\bibliography{biblio}

\begin{thebibliography}{40}
\expandafter\ifx\csname natexlab\endcsname\relax\def\natexlab#1{#1}\fi
\providecommand{\url}[1]{\texttt{#1}}
\providecommand{\href}[2]{#2}
\providecommand{\path}[1]{#1}
\providecommand{\DOIprefix}{doi:}
\providecommand{\ArXivprefix}{arXiv:}
\providecommand{\URLprefix}{URL: }
\providecommand{\Pubmedprefix}{pmid:}
\providecommand{\doi}[1]{\href{http://dx.doi.org/#1}{\path{#1}}}
\providecommand{\Pubmed}[1]{\href{pmid:#1}{\path{#1}}}
\providecommand{\bibinfo}[2]{#2}
\ifx\xfnm\relax \def\xfnm[#1]{\unskip,\space#1}\fi
\bibitem[{Anis(2013)}]{anis}
\bibinfo{author}{Anis, M.}, \bibinfo{year}{2013}.
\newblock \bibinfo{title}{A family of tests for exponentiality against ifr
  alternatives}.
\newblock \bibinfo{journal}{Journal of Statistical Planning and Inference}
  \bibinfo{volume}{143}, \bibinfo{pages}{1409--1415}.
\bibitem[{Anis and Basu(2014)}]{anis2014}
\bibinfo{author}{Anis, M.}, \bibinfo{author}{Basu, K.}, \bibinfo{year}{2014}.
\newblock \bibinfo{title}{Tests for exponentiality against {NBUE} alternatives:
  a {M}onte {C}arlo comparison}.
\newblock \bibinfo{journal}{Journal of Statistical Computation and Simulation}
  \bibinfo{volume}{84}, \bibinfo{pages}{231--247}.
\bibitem[{Arab et~al.(2025)Arab, Lando and Oliveira}]{landoJAP}
\bibinfo{author}{Arab, I.}, \bibinfo{author}{Lando, T.},
  \bibinfo{author}{Oliveira, P.}, \bibinfo{year}{2025}.
\newblock \bibinfo{title}{Inequalities and bounds for expected order statistics
  from convex ordered families}.
\newblock \bibinfo{journal}{Journal of Applied Probability} ,
  \bibinfo{pages}{1--19}.
\bibitem[{Barlow et~al.(1963)Barlow, Marshall and Proschan}]{barlow1963}
\bibinfo{author}{Barlow, R.E.}, \bibinfo{author}{Marshall, A.W.},
  \bibinfo{author}{Proschan, F.}, \bibinfo{year}{1963}.
\newblock \bibinfo{title}{Properties of probability distributions with monotone
  hazard rate}.
\newblock \bibinfo{journal}{The Annals of Mathematical Statistics} ,
  \bibinfo{pages}{375--389}.
\bibitem[{Barlow and Van~Zwet(1969)}]{barlowzwet}
\bibinfo{author}{Barlow, R.E.}, \bibinfo{author}{Van~Zwet, W.R.},
  \bibinfo{year}{1969}.
\newblock \bibinfo{title}{Asymptotic properties of isotonic estimators for the
  generalized failure rate function. {P}art 1: strong consistency}.
\newblock \bibinfo{type}{Technical Report}. California Univ Berkeley
  {O}perations {R}esearch center.
\bibitem[{Beare(2021)}]{beare2021}
\bibinfo{author}{Beare, B.K.}, \bibinfo{year}{2021}.
\newblock \bibinfo{title}{Least favourability of the uniform distribution for
  tests of the concavity of a distribution function}.
\newblock \bibinfo{journal}{Stat} \bibinfo{volume}{10}, \bibinfo{pages}{e376}.
\bibitem[{Bickel and Doksum(1969)}]{bickel}
\bibinfo{author}{Bickel, P.J.}, \bibinfo{author}{Doksum, K.A.},
  \bibinfo{year}{1969}.
\newblock \bibinfo{title}{Tests for monotone failure rate based on normalized
  spacings}.
\newblock \bibinfo{journal}{The Annals of Mathematical Statistics} ,
  \bibinfo{pages}{1216--1235}.
\bibitem[{Billingsley(2013)}]{billingsley}
\bibinfo{author}{Billingsley, P.}, \bibinfo{year}{2013}.
\newblock \bibinfo{title}{Convergence of probability measures}.
\newblock \bibinfo{publisher}{John Wiley \& Sons}.
\bibitem[{Bingham et~al.(1989)Bingham, Goldie and Teugels}]{bingham}
\bibinfo{author}{Bingham, N.H.}, \bibinfo{author}{Goldie, C.M.},
  \bibinfo{author}{Teugels, J.L.}, \bibinfo{year}{1989}.
\newblock \bibinfo{title}{Regular variation}. volume~\bibinfo{volume}{27}.
\newblock \bibinfo{publisher}{Cambridge university press}.
\bibitem[{Block et~al.(1998)Block, Savits and Singh}]{block}
\bibinfo{author}{Block, H.W.}, \bibinfo{author}{Savits, T.H.},
  \bibinfo{author}{Singh, H.}, \bibinfo{year}{1998}.
\newblock \bibinfo{title}{The reversed hazard rate function}.
\newblock \bibinfo{journal}{Probability in the Engineering and informational
  Sciences} \bibinfo{volume}{12}, \bibinfo{pages}{69--90}.
\bibitem[{Bryson(1974)}]{bryson}
\bibinfo{author}{Bryson, M.C.}, \bibinfo{year}{1974}.
\newblock \bibinfo{title}{Heavy-tailed distributions: properties and tests}.
\newblock \bibinfo{journal}{Technometrics} \bibinfo{volume}{16},
  \bibinfo{pages}{61--68}.
\bibitem[{Carolan(2002)}]{carolan}
\bibinfo{author}{Carolan, C.A.}, \bibinfo{year}{2002}.
\newblock \bibinfo{title}{The least concave majorant of the empirical
  distribution function}.
\newblock \bibinfo{journal}{Canadian Journal of Statistics}
  \bibinfo{volume}{30}, \bibinfo{pages}{317--328}.
\bibitem[{Chen et~al.(2024)Chen, Embrechts and Wang}]{chen2024}
\bibinfo{author}{Chen, Y.}, \bibinfo{author}{Embrechts, P.},
  \bibinfo{author}{Wang, R.}, \bibinfo{year}{2024}.
\newblock \bibinfo{title}{An unexpected stochastic dominance: Pareto
  distributions, dependence, and diversification}.
\newblock \bibinfo{journal}{Operations Research} .
\bibitem[{Chen and Shneer(2024)}]{chen2024stochastic}
\bibinfo{author}{Chen, Y.}, \bibinfo{author}{Shneer, S.}, \bibinfo{year}{2024}.
\newblock \bibinfo{title}{Stochastic dominance for super heavy-tailed random
  variables}.
\newblock \bibinfo{journal}{arXiv preprint arXiv:2408.15033} .
\bibitem[{Gijbels and Heckman(2004)}]{gibels}
\bibinfo{author}{Gijbels, I.}, \bibinfo{author}{Heckman, N.},
  \bibinfo{year}{2004}.
\newblock \bibinfo{title}{Nonparametric testing for a monotone hazard function
  via normalized spacings}.
\newblock \bibinfo{journal}{Journal of Nonparametric Statistics}
  \bibinfo{volume}{16}, \bibinfo{pages}{463--477}.
\bibitem[{Grenander(1956)}]{grenander}
\bibinfo{author}{Grenander, U.}, \bibinfo{year}{1956}.
\newblock \bibinfo{title}{On the theory of mortality measurement: part {II}}.
\newblock \bibinfo{journal}{Scandinavian Actuarial Journal}
  \bibinfo{volume}{1956}, \bibinfo{pages}{125--153}.
\bibitem[{Groeneboom and Jongbloed(2012)}]{groeneboom2012}
\bibinfo{author}{Groeneboom, P.}, \bibinfo{author}{Jongbloed, G.},
  \bibinfo{year}{2012}.
\newblock \bibinfo{title}{Isotonic ${L}_2$-projection test for local
  monotonicity of a hazard}.
\newblock \bibinfo{journal}{Journal of Statistical Planning and Inference}
  \bibinfo{volume}{142}, \bibinfo{pages}{1644--1658}.
\bibitem[{Groeneboom and Jongbloed(2014)}]{groeneboom2014}
\bibinfo{author}{Groeneboom, P.}, \bibinfo{author}{Jongbloed, G.},
  \bibinfo{year}{2014}.
\newblock \bibinfo{title}{Nonparametric estimation under shape constraints}.
\newblock \bibinfo{number}{38}, \bibinfo{publisher}{Cambridge University
  Press}.
\bibitem[{de~Haan and Ferreira(2006)}]{haan}
\bibinfo{author}{de~Haan, L.}, \bibinfo{author}{Ferreira, A.},
  \bibinfo{year}{2006}.
\newblock \bibinfo{title}{Extreme value theory: an introduction}.
  volume~\bibinfo{volume}{3}.
\newblock \bibinfo{publisher}{Springer}.
\bibitem[{Hall and Van~Keilegom(2005)}]{hall2005}
\bibinfo{author}{Hall, P.}, \bibinfo{author}{Van~Keilegom, I.},
  \bibinfo{year}{2005}.
\newblock \bibinfo{title}{Testing for monotone increasing hazard rate}.
\newblock \bibinfo{journal}{Annals of Statistics} \bibinfo{volume}{33},
  \bibinfo{pages}{1109--1137}.
\bibitem[{Lando(2023)}]{lando2023}
\bibinfo{author}{Lando, T.}, \bibinfo{year}{2023}.
\newblock \bibinfo{title}{Testing departures from the increasing hazard rate
  property}.
\newblock \bibinfo{journal}{Statistics \& Probability Letters}
  \bibinfo{volume}{193}, \bibinfo{pages}{109736}.
\bibitem[{Lando et~al.(2024)Lando, Arab and Eduardo~Oliveira}]{lando2024}
\bibinfo{author}{Lando, T.}, \bibinfo{author}{Arab, I.},
  \bibinfo{author}{Eduardo~Oliveira, P.}, \bibinfo{year}{2024}.
\newblock \bibinfo{title}{Nonparametric inference about increasing odds rate
  distributions}.
\newblock \bibinfo{journal}{Journal of Nonparametric Statistics}
  \bibinfo{volume}{36}, \bibinfo{pages}{435--454}.
\bibitem[{Lando et~al.(2022)Lando, Arab and Oliveira}]{odds}
\bibinfo{author}{Lando, T.}, \bibinfo{author}{Arab, I.},
  \bibinfo{author}{Oliveira, P.E.}, \bibinfo{year}{2022}.
\newblock \bibinfo{title}{Properties of increasing odds rate distributions with
  a statistical application}.
\newblock \bibinfo{journal}{Journal of Statistical Planning and Inference}
  \bibinfo{volume}{221}, \bibinfo{pages}{313--325}.
\bibitem[{Lando et~al.(2023)Lando, Arab and Oliveira}]{landotransform}
\bibinfo{author}{Lando, T.}, \bibinfo{author}{Arab, I.},
  \bibinfo{author}{Oliveira, P.E.}, \bibinfo{year}{2023}.
\newblock \bibinfo{title}{Transform orders and stochastic monotonicity of
  statistical functionals}.
\newblock \bibinfo{journal}{Scandinavian Journal of Statistics}
  \bibinfo{volume}{50}, \bibinfo{pages}{1183--1200}.
\bibitem[{Marshall and Olkin(2007)}]{LifeDist}
\bibinfo{author}{Marshall, A.W.}, \bibinfo{author}{Olkin, I.},
  \bibinfo{year}{2007}.
\newblock \bibinfo{title}{Life Distributions}.
\newblock \bibinfo{publisher}{Springer, New York}.
\bibitem[{Mason(1982)}]{mason}
\bibinfo{author}{Mason, D.M.}, \bibinfo{year}{1982}.
\newblock \bibinfo{title}{Some characterizations of strong laws for linear
  functions of order statistics}.
\newblock \bibinfo{journal}{The Annals of Probability} ,
  \bibinfo{pages}{1051--1057}.
\bibitem[{Mitra and Anis(2008)}]{mitra2008}
\bibinfo{author}{Mitra, M.}, \bibinfo{author}{Anis, M.}, \bibinfo{year}{2008}.
\newblock \bibinfo{title}{An {L}-statistic approach to a test of exponentiality
  against {IFR} alternatives}.
\newblock \bibinfo{journal}{Journal of Statistical Planning and Inference}
  \bibinfo{volume}{138}, \bibinfo{pages}{3144--3148}.
\bibitem[{M{\"u}ller(2024)}]{muller2024}
\bibinfo{author}{M{\"u}ller, A.}, \bibinfo{year}{2024}.
\newblock \bibinfo{title}{Some remarks on the effect of risk sharing and
  diversification for infinite mean risks}.
\newblock \bibinfo{journal}{arXiv:2411.10139} .
\bibitem[{Oja(1981)}]{oja}
\bibinfo{author}{Oja, H.}, \bibinfo{year}{1981}.
\newblock \bibinfo{title}{On location, scale, skewness and kurtosis of
  univariate distributions}.
\newblock \bibinfo{journal}{Scandinavian Journal of statistics} ,
  \bibinfo{pages}{154--168}.
\bibitem[{Proschan and Pyke(1967)}]{proschan1967}
\bibinfo{author}{Proschan, F.}, \bibinfo{author}{Pyke, R.},
  \bibinfo{year}{1967}.
\newblock \bibinfo{title}{Tests for monotone failure rate}, in:
  \bibinfo{booktitle}{Fifth Berkley Symposium}, pp. \bibinfo{pages}{293--313}.
\bibitem[{Resnick(2008)}]{resnick}
\bibinfo{author}{Resnick, S.I.}, \bibinfo{year}{2008}.
\newblock \bibinfo{title}{Extreme values, regular variation, and point
  processes}. volume~\bibinfo{volume}{4}.
\newblock \bibinfo{publisher}{Springer Science \& Business Media}.
\bibitem[{Robertson et~al.(1988)Robertson, Wright and Dykstra}]{robertson1988}
\bibinfo{author}{Robertson, T.}, \bibinfo{author}{Wright, F.T.},
  \bibinfo{author}{Dykstra, R.L.}, \bibinfo{year}{1988}.
\newblock \bibinfo{title}{Order restricted statistical inference}.
\newblock \bibinfo{publisher}{Wiley}.
\bibitem[{Serfling(2009)}]{serfling}
\bibinfo{author}{Serfling, R.J.}, \bibinfo{year}{2009}.
\newblock \bibinfo{title}{Approximation theorems of mathematical statistics}.
\newblock \bibinfo{publisher}{John Wiley \& Sons}.
\bibitem[{Shaked and Shantikumar(2007)}]{shaked}
\bibinfo{author}{Shaked, M.}, \bibinfo{author}{Shantikumar, J.G.},
  \bibinfo{year}{2007}.
\newblock \bibinfo{title}{Stochastic Orders}.
\newblock \bibinfo{publisher}{Springer, New York}.
\bibitem[{Staudte(2017a)}]{staudte2017}
\bibinfo{author}{Staudte, R.G.}, \bibinfo{year}{2017}a.
\newblock \bibinfo{title}{Inference for quantile measures of kurtosis,
  peakedness, and tail weight}.
\newblock \bibinfo{journal}{Communications in Statistics-Theory and Methods}
  \bibinfo{volume}{46}, \bibinfo{pages}{3148--3163}.
\bibitem[{Staudte(2017b)}]{staudte}
\bibinfo{author}{Staudte, R.G.}, \bibinfo{year}{2017}b.
\newblock \bibinfo{title}{The shapes of things to come: Probability density
  quantiles}.
\newblock \bibinfo{journal}{Statistics} \bibinfo{volume}{51},
  \bibinfo{pages}{782--800}.
\bibitem[{Van~Zwet(1964)}]{vanzwet1964}
\bibinfo{author}{Van~Zwet, W.R.}, \bibinfo{year}{1964}.
\newblock \bibinfo{title}{Convex transformations of random variables}.
\newblock \bibinfo{journal}{MC Tracts} .
\bibitem[{Van~Zwet(1980)}]{zwet1980}
\bibinfo{author}{Van~Zwet, W.R.}, \bibinfo{year}{1980}.
\newblock \bibinfo{title}{A strong law for linear functions of order
  statistics}.
\newblock \bibinfo{journal}{The Annals of Probability} \bibinfo{volume}{5},
  \bibinfo{pages}{986--990}.
\bibitem[{Wellner(1977)}]{wellner}
\bibinfo{author}{Wellner, J.A.}, \bibinfo{year}{1977}.
\newblock \bibinfo{title}{A {G}livenko-{C}antelli theorem and strong laws of
  large numbers for functions of order statistics}.
\newblock \bibinfo{journal}{The Annals of Statistics} \bibinfo{volume}{5},
  \bibinfo{pages}{473--480}.
\bibitem[{Zimmer et~al.(1998)Zimmer, Wang and Pathak}]{zimmer}
\bibinfo{author}{Zimmer, W.J.}, \bibinfo{author}{Wang, Y.},
  \bibinfo{author}{Pathak, P.K.}, \bibinfo{year}{1998}.
\newblock \bibinfo{title}{Log-odds rate and monotone log-odds rate
  distributions}.
\newblock \bibinfo{journal}{Journal of quality technology}
  \bibinfo{volume}{30}, \bibinfo{pages}{376--385}.

\end{thebibliography}

\end{document}